\documentclass[12pt,a4paper]{amsart}

\usepackage[margin=1in]{geometry}
\allowdisplaybreaks[1]

\usepackage{amsmath}
\usepackage{amsfonts}
\usepackage{graphicx}
\usepackage{framed}
\usepackage{amssymb}
\usepackage{esvect}
\usepackage{xcolor}
\usepackage{cite}

\usepackage[backref=page]{hyperref}

\renewcommand*{\backref}[1]{}
\renewcommand*{\backrefalt}[4]{%
    \ifcase #1 (Not cited.)%
    \or        (Cited on page~#2.)%
    \else      (Cited on pages~#2.)%
    \fi}

\usepackage{etex}
\usepackage{latexsym}
\usepackage[english]{babel}
\usepackage[utf8x]{inputenc}

\def \N {\mathbb{N}}
\def \R {\mathbb{R}}
\def \e {\varepsilon}
\renewcommand{\epsilon}{\varepsilon}

\usepackage{amsthm}
\theoremstyle{definition}
\newtheorem{definition}{Definition}[section]

\theoremstyle{plain}
\newtheorem{theorem}[definition]{Theorem}

\newtheorem{lemma}[definition]{Lemma}
\newtheorem{corollary}[definition]{Corollary}

\renewcommand{\leq}{\leqslant}

\renewcommand{\ge}{\geqslant}
\renewcommand{\le}{\leqslant}

\usepackage{mathtools}
\mathtoolsset{showonlyrefs}
\usepackage{mathrsfs}

\numberwithin{equation}{section}

\newcommand{\vertiii}[1]{{\left\vert\kern-0.25ex\left\vert\kern-0.25ex\left\vert #1 
    \right\vert\kern-0.25ex\right\vert\kern-0.25ex\right\vert}}

\usepackage[alphabetic]{amsrefs}

\begin{document}

\begin{abstract} 
This article investigates a mathematical model for bushfire propagation, focusing on the existence and properties of translating solutions. We obtain quantitative bounds on the environmental diffusion coefficient and ignition kernels, identifying conditions under which fires either propagate across the entire region or naturally extinguish. 

Our analysis also reveals that vertically translating solutions do not exist, whereas traveling wave solutions with a front moving at any prescribed velocity always exist for kernels that are either of mild intensity or short range. These traveling waves exhibit unbounded
profiles. 
\end{abstract}
 
\title[Self-sustaining traveling fronts for bushfires]{Self-sustaining traveling fronts \\ for a model related to bushfires}

\author[S. Dipierro]{Serena Dipierro}
\address{S. D., 
Department of Mathematics and Statistics,
University of Western Australia,
35~Stirling Highway, WA 6009 Crawley, Australia. }

\email{serena.dipierro@uwa.edu.au}

\author[E. Valdinoci]{Enrico Valdinoci}
\address{E. V., 
Department of Mathematics and Statistics,
University of Western Australia,
35~Stirling Highway, WA 6009 Crawley, Australia. }

\email{enrico.valdinoci@uwa.edu.au}

\author[G. Wheeler]{Glen Wheeler}
\address{G. W.,
School of Mathematics and Applied Statistics,
University of Wollongong,
Northfields Avenue, NSW 2500 Wollongong, Australia. }

\email{glenw@uow.edu.au}

\author[V.-M. Wheeler]{Valentina-Mira Wheeler}
\address{V.-M. W.,
School of Mathematics and Applied Statistics,
University of Wollongong,
Northfields Avenue, NSW 2500 Wollongong, Australia. }

\email{vwheeler@uow.edu.au}

\thanks{Supported by Australian Research Council DP250101080 and DE190100379.}

\subjclass[2020]{35C07; 35K10; 45H05; 92-10.}

\keywords{bushfire models; evolution equation; moving fronts; traveling waves.}

\maketitle

\section{Introduction} 
Understanding the dynamics of bushfire propagation is crucial for predicting
fire behavior and implementing effective mitigation strategies. Mathematical
modeling provides a rigorous framework to analyze the interplay between
environmental factors and fire spread. This study focuses on a mathematical
model describing bushfire propagation, with particular emphasis on the
existence and properties of self-similar solutions. Typically, in nonlinear
dynamics, self-sustaining structures play a fundamental role in understanding
long-term behavior and pattern formation. Their simple structure and
predictability is often helpful to describe persistent phenomena.

Specifically, in bushfire modeling, translating solutions represent
steady-state firefronts moving through an environment and the analysis of these
solutions is crucial to identify conditions under which a fire sustains itself
or extinguishes. {F}rom the phenomenological standpoint,
these patterns can also help to quantify how parameters like diffusion,
ignition thresholds, and fuel availability influence fire spread.

In this paper, by deriving quantitative bounds on the environmental diffusion coefficient and ignition kernels, we establish conditions that dictate whether a fire will sustain itself indefinitely or eventually extinguish. Our findings demonstrate that while vertically translating solutions (which correspond physically to fires whose temperature is a linear function of time)
are not possible, traveling wave solutions exist for any prescribed velocity, exhibiting divergent profiles. 

These results contribute to the broader understanding of fire dynamics by offering precise conditions for sustained propagation and extinction. They also provide insights into the robustness of firefronts under localized disturbances, which has direct implications for fire management and predictive modeling in real-world scenarios.\medskip

Diving into the specific features\footnote{We stress that many bushfire models are available in the literature, each with specific characteristics and different ranges of applicability. Providing a full account of all possible models is beyond the scope of this work. Clearly, when proving a precise mathematical statement, it is not feasible to treat all conceivable models simultaneously. Accordingly, we follow the standard practice of focusing on a single model, namely the one introduced in~\cite{PAPER1}. To the best of our knowledge, the present paper constitutes the first thorough analysis of self-similar solutions in the bushfire context.}
 of this work, we
use here as our primary tool the model recently proposed in~\cite{PAPER1}.
First, we consider some straightforward tests of the model motivated by physical intuition.
	For instance, a fire (solution to the model) should not spontaneously ignite without being hot enough for ignition to begin, either from the initial condition or from the boundary condition -- that this is indeed the case is confirmed by Lemma~\ref{NECE}.
	The main tool we use for proving this, which is interesting in its own right, is a comparison principle or ordering property of solutions (Lemma~\ref{COMPAPI}).
	Further intuitive knowledge of fire behavior includes the complete spread of a fire from a single point or the complete suppression of a fire depending on the balance of diffusion parameters and inition condition.
	This is the subject of Theorem~\ref{SOPTH} and Theorem~\ref{SOPTH:2} respectively.
	
	Next, we consider the important operational question of the propagation speed of a fire front ignited at a spatial boundary. This information may be useful for estimating evacuation times. We provide qualitative estimates of the propagation speed in Theorem~\ref{DBDARYEFF}.
	
	Finally, we systematically study solutions that take a relatively simple shape, or whose dynamics are straightforward to describe.
	For example, fires that move in a given direction without changing shape, or whose temperature only increases constantly with respect to time.
	We show the non-existence of fires that steadily increase in temperature (vertical translating solutions, Theorem~\ref{ABSE}),
	non-existence of fires that move self-similarly (Theorem~\ref{sojdcmvvb-rw-SeeDeqwdf}), 
	and the existence of infinite fire waves moving through a region (traveling wave solutions, Theorems~\ref{TRAVE} and~\ref{TRAVE.BIS}).
	The waves are always unbounded (as shown in Theorem~\ref{ABBOU}).
	
	In the next section we state precisely these results and give a brief discussion for each of them.

\section*{Funding declaration}

All authors were partially supported by ARC (Australian Research Council) DP250101080, and the fourth author additionally partially supported by ARC DE190100379.
They are grateful for this support.

\section{Main results}
\subsection{Mathematical framework and consistency checks}
Suppose that the environmental temperature~$u$ at a spatial location~$x$ at
time~$t$ is described by the evolution equation
\begin{equation}\label{MAIN:EQ:pa3}\begin{split}
\partial_t u(x,t)=c\Delta u(x,t)+\int_{\Omega}\big(u(y,t)-\Theta\big)_+\,K(x,y)\,dy,\end{split}
\end{equation}
where~$c\in(0,+\infty)$ is the diffusion coefficient and~$\Theta\in\R$ the ignition temperature.\footnote{{Up to replacing~$u$
with~$u-\Theta$, the parameter~$\Theta$ can be scaled away, thereby simplifying equation~\eqref{MAIN:EQ:pa3}. In this formulation, the ignition temperature becomes zero. However, throughout this paper we retain the parameter~$\Theta$, both to avoid confusion and for consistency with other works, where the ignition temperature may depend on time and position and may not be smooth, making such a scaling less convenient.}}
Here and throughout the paper, we use the positive part notation: for every~$r\in\R$, we let~$r_+:=\max\{r,0\}$.

We assume that the interaction kernel~$K$ is nonnegative and integrable, with
\begin{equation}\label{UNIKe}
\sup_{x\in\Omega}\int_{\Omega}K(x,y)\,dy<+\infty.\end{equation}

We suppose that the evolution equation~\eqref{MAIN:EQ:pa3} takes place in a bounded and smooth domain~$\Omega\subset\R^n$, with a given initial condition~$u(x,0)$ at time~$t=0$
and a given Dirichlet datum, possibly depending on time~$t$, assigned along~$\partial\Omega$.

For simplicity, we focus\footnote{Under natural assumptions, equation~\eqref{MAIN:EQ:pa3}
admits classical solutions, see~\cite{MR4968074}. For example, for bounded kernels, analytic semigroup theory (see e.g.~\cite[Theorems~3.3.3 and~3.3.4]{MR610244}) gives the existence of mild solutions, and, for smooth kernels, one can obtain classical solutions by differentiating the equation (see e.g.~\cite[Theorem~3.1 on page~196]{MR710486}
 and~\cite[Theorem~6 on page~65]{MR181836}). Another interesting case occurs when~$K$ is the Green Function of the Laplace operator, since, defining
$$U(x):=\int_{\Omega} (u(y, t) - \Theta)_+ K(x, y) \, dy,$$
this setting corresponds to the parabolic-elliptic system 
$$\begin{cases} \partial_t u = c\Delta u + U \\ -\Delta U = (u - \Theta)_+, \end{cases}$$
which arises in mathematical biology and chemotaxis.} here on classical solutions, namely we suppose that for every~$x\in\Omega$ the map~$(0,+\infty)\ni t\mapsto u(x,t)$ is continuously differentiable, that
for every~$t\in(0,+\infty)$ the map~$\Omega\ni x\mapsto u(x,t)$ is twice continuously differentiable, and that~$u$ is continuous in~$\overline\Omega\times[0,+\infty)$.

A useful observation is that solutions of~\eqref{MAIN:EQ:pa3} satisfy a natural ordering property:

\begin{lemma}[Comparison Principle] \label{COMPAPI}
Let~$u$ and~$v$ be such that
\begin{equation}\label{MAIN:EQ:pa3:sSop}\begin{split}
&\partial_t u(x,t)\le c\Delta u(x,t)+\int_{\Omega}\big(u(y,t)-\Theta\big)_+\,K(x,y)\,dy\\
{\mbox{and }}\quad&\partial_t v(x,t)\ge c\Delta v(x,t)+\int_{\Omega}\big(v(y,t)-\Theta\big)_+\,K(x,y)\,dy
\end{split}
\end{equation}
in~$\Omega\times(0,+\infty)$,
with~$u(x,0)\le v(x,0)$ for all~$x\in\Omega$ and~$u(x,t)\le v(x,t)$ for all~$x\in\partial\Omega$ and~$t\in[0,+\infty)$.

Then, $u(x,t)\le v(x,t)$ for all~$x\in\Omega$ and~$t\in[0,+\infty)$.
\end{lemma}

For practical purposes, it is also useful to distinguish situations in which a fire takes place.
To this end, we say that a solution of~\eqref{MAIN:EQ:pa3} is ``burning'' \label{BURSO}
if it takes at least some values above the ignition temperature,
i.e. there exist~$x_0\in\Omega$ and~$t_0\in[0,+\infty)$ such that~$u(x_0,t_0)>\Theta$.

As a consistency check, which follows as a byproduct of the Comparison Principle in Lemma~\ref{COMPAPI},
let us point out that burning solutions can only be produced by burning initial or boundary data:

\begin{lemma}[Necessity of the initial ignition] \label{NECE}
Let~$u$ be a solution of~\eqref{MAIN:EQ:pa3} such that~$u(x,0)\le\Theta$ for all~$x\in\Omega$ and~$u(x,t)\le\Theta$ for all~$x\in\partial\Omega$ and~$t\in[0,+\infty)$.

Then, $u(x,t)\le\Theta$ for all~$x\in\Omega$ and~$t\in[0,+\infty)$.
\end{lemma}

\subsection{Fire invasion or extinction}
The Comparison Principle in Lemma~\ref{COMPAPI} possesses further interesting practical consequences,
as showcased\footnote{As customary, $B_r$ denotes
the ball of radius~$r$ centered at the origin.}
by the next result:

\begin{theorem}[Description of a fire invading the whole region]\label{SOPTH}
Let
\begin{equation}\label{CPDFF:DIZ}
\lambda_0>\Theta>0.\end{equation}
Assume that, for all~$r>0$,
\begin{equation}\label{CPDFF:DIZ0} \inf_{x\in B_1}
\int_{B_r}K(x,y)\,dy>0.\end{equation}

Let~$u$ be a solution of 
\begin{equation}\label{fuwekkfg12345eajhdgahks}
\partial_t u(x,t)=c\Delta u(x,t)+\int_{B_1}\big(u(y,t)-\Theta\big)_+\,K(x,y)\,dy,\end{equation}
with~$u(x,t)=0$ for all~$x\in\partial B_1$ and~$t\in[0,+\infty)$.

Suppose that
\begin{equation}\label{CPDFF:DIZ2}u(x,0)\ge
\lambda_0 (1-|x|^2)
.\end{equation}

Then, there exists~$\overline{c}>0$,
depending only on~$n$, $\lambda_0 $, $\Theta$, and~$K$, such that, if
\begin{equation}\label{CPDFF:DIZ2.b}c\le \overline{c},\end{equation}
we have that,
for all~$x\in B_1$,
\begin{equation}\label{CPDFF:DIZ2.c}\lim_{t\to+\infty}u(x,t)=+\infty.\end{equation}

More precisely, under the above assumptions,
there exists~$\alpha>0$, depending only on~$\lambda_0 $, $\Theta$, and~$K$,
such that, for all~$x\in B_1$ and~$t\in[0,+\infty)$,
\begin{equation}\label{CPDFF:DIZ2.d} u(x,t)\ge\lambda_0 e^{\alpha t} (1-|x|^2).\end{equation}
\end{theorem}

We emphasize that condition~$\Theta>0$
in~\eqref{CPDFF:DIZ} ensures that the boundary values of the domain remain below the ignition temperature, namely the fire is not a result of boundary effects
(to be compared with the forthcoming Theorem~\ref{DBDARYEFF}). 

Furthermore, conditions~\eqref{CPDFF:DIZ} and~\eqref{CPDFF:DIZ2} establish that~$u(0,0)=\lambda_0>\Theta$, indicating that the center of the domain exceeds the ignition temperature and thus, roughly speaking, the fire originates primarily from the center of the domain.

Moreover, condition~\eqref{CPDFF:DIZ0} stipulates that the interaction term propagating the fire is sufficiently active throughout the entire region (a condition satisfied, for instance, by all Gaussian-type interaction kernels). In contrast, condition~\eqref{CPDFF:DIZ2.b} requires that the diffusion coefficient is small enough to prevent heat from being dissipated too quickly by the environment (to be compared with
the forthcoming Theorem~\ref{SOPTH:2}).

In this scenario, the conclusion of Theorem~\ref{SOPTH}, as detailed in~\eqref{CPDFF:DIZ2.c}, indicates that the entire domain\footnote{Needless to say, 
as usual in mathematical models,
when blow-up occurs, as in~\eqref{CPDFF:DIZ2.c}, no claim of physical realism is made at extreme values. In our context, this
``blow-up in infinite time'' simply means that for any given point in the domain, there exists a time at which it burns. Since this time depends on the spatial point, appropriate care is required to formulate the statement precisely. 
Quantitative bounds, such as~\eqref{CPDFF:DIZ2.d}, are useful in these scenarios if one wishes to modify the model to describe the aftermath of the burning process, since they indicate the stage at which modifications need to be implemented.} will be engulfed in fire. More strikingly, as highlighted in~\eqref{CPDFF:DIZ2.d}, the environmental temperature will increase at an exponential rate.
\medskip

An interesting counterpart of Theorem~\ref{SOPTH} is provided by the following result:

\begin{theorem}[Description of a fire being extinguished by environmental thermal diffusion]\label{SOPTH:2} Assume that~\eqref{CPDFF:DIZ} is satisfied and that
\begin{equation}\label{ihsdcknvAS234MSsijsmdvlpsd}C:=
\sup_{x\in B_1}\int_{B_1}(1-|y|^2)\,K(x,y)\,dy<+\infty.\end{equation}

Let~$u$ be a solution of 
$$ \partial_t u(x,t)=c\Delta u(x,t)+\int_{B_1}\big(u(y,t)-\Theta\big)_+\,K(x,y)\,dy,$$
with~$u(x,t)=0$ for all~$x\in\partial B_1$ and~$t\in[0,+\infty)$.

Suppose that
\begin{equation}\label{CPDFF:DIZ2:2}u(x,0)\le
\lambda_0 (1-|x|^2)
.\end{equation}

Then, if
\begin{equation}\label{CPDFF:DIZ2.b:2}c>\frac{C}{2n},\end{equation}
we have that,
for all~$x\in B_1$,
\begin{equation}\label{CPDFF:DIZ2.c:2}\lim_{t\to+\infty}u(x,t)\le0.\end{equation}

More precisely, under the above assumptions,
for all~$x\in B_1$ and~$t\in[0,+\infty)$,
\begin{equation}\label{CPDFF:DIZ2.d:2} u(x,t)\le\lambda_0 e^{-(2nc-C) t} (1-|x|^2).\end{equation}
\end{theorem}

The significance of Theorem~\ref{SOPTH:2} lies in its provision of quantitative bounds for a fire ignited at the center of a domain to extinguish solely due to the thermal diffusivity of the environment. According to~\eqref{CPDFF:DIZ2.b:2}, this requires the thermal diffusivity to be sufficiently high relative to the interaction kernel. Note that~\eqref{ihsdcknvAS234MSsijsmdvlpsd} is automatically satisfied, for instance, if~$K$ is bounded. 

Furthermore, the conclusion of Theorem~\ref{SOPTH:2}, stated in~\eqref{CPDFF:DIZ2.c:2}, ensures that the entire domain eventually remains below the ignition temperature. In fact, as elaborated in~\eqref{CPDFF:DIZ2.d:2}, this occurs at an exponential rate.\medskip

In comparison with Theorem~\ref{SOPTH}, that describes a scenario in which the fire is initiated at the center
of the domain, it is also interesting to analyze the situation in which the fire is started at the boundary and propagates inside the domain. This is described in the following result:

\begin{theorem}[Boundary ignition]\label{DBDARYEFF}
Let
\begin{equation}\label{BDw3rA:CPDFF:D2sfIZ}
\overline\Theta>\Theta\qquad{\mbox{and}}\qquad \beta>\overline\Theta-\Theta.\end{equation}

Let also
\begin{equation}\label{CASDF:d2}
\alpha:=2nc (\overline\Theta-\Theta)\in(0,+\infty)
\end{equation}
and
\begin{equation}\label{CASDF:d3}
t_\star:=\frac{\beta-\overline\Theta+\Theta}\alpha.\end{equation}

Let~$u$ be a solution of 
\begin{equation}\label{weuoyr348ofnkdel0976}
\partial_t u(x,t)=c\Delta u(x,t)+\int_{B_1}\big(u(y,t)-\Theta\big)_+\,K(x,y)\,dy,
\end{equation}
with~$u(x,t)=\overline\Theta$ for all~$x\in\partial B_1$ and~$t\in[0,+\infty)$.

Suppose that
\begin{equation}\label{CPDFF:	qweDIZ2:BDCO}u(x,0)\ge
\overline\Theta-\beta(1-|x|^2)
.\end{equation}

Then, for all~$x\in B_1$,
\begin{equation}\label{CPDFF:DIZ2.c.bdary}
u(x,t_\star)\ge\Theta.\end{equation}

More precisely, under the above assumptions,
for all~$x\in B_1$ and~$t\in[0,t_\star]$,
\begin{equation}\label{CPDFF:DIZ2.dwafsger2} u(x,t)\ge
\overline\Theta-(\beta-\alpha t)
(1-|x|^2).\end{equation}
\end{theorem}

We point out that Theorem~\ref{DBDARYEFF} relies only on the diffusion term of equation~\eqref{weuoyr348ofnkdel0976} and remains valid\footnote{We observe that the role of condition~\eqref{CPDFF: qweDIZ2:BDCO} is solely to provide an explicit lower bound as in~\eqref{CPDFF:DIZ2.dwafsger2}. Indeed, the solution of the heat equation can be written as
$$ u(x,t)=\overline\Theta+\sum_{k=1}^{+\infty} c_k\,e^{-c\lambda_k t}\,\phi_k(x),$$
where the coefficients~$c_k$
depend on the initial condition, and~$\phi_k$
is the eigenfunction of the (negative) Dirichlet Laplacian associated with the (positive) eigenvalue~$\lambda_k$.
As a consequence, $u(x,t)$ converges to~$\overline\Theta$
as~$t\to+\infty$, regardless of the initial temperature distribution. Hence, even in the absence of~\eqref{CPDFF: qweDIZ2:BDCO}, the available region is eventually burned by the fire. This effect is further enhanced by the integral term arising from the ignition kernel, while explicit quantitative bounds can be derived from~\eqref{CPDFF: qweDIZ2:BDCO}.}
even for the heat equation (corresponding to the case~$K:=0$). This, in particular, underscores a structural difference with Theorem~\ref{SOPTH}, which instead crucially hinges on the interaction term of equation~\eqref{fuwekkfg12345eajhdgahks}
and indeed requires the diffusion coefficient to be sufficiently small.

We also observe that, on the one hand, condition~\eqref{BDw3rA:CPDFF:D2sfIZ} entails that the boundary of the
domain in Theorem~\ref{DBDARYEFF} lies above the ignition temperature (since, for~$x\in\partial B_1$,
we have that~$u(x,t)=\overline\Theta>\Theta$).

On the other hand, conditions~\eqref{BDw3rA:CPDFF:D2sfIZ}
and~\eqref{CPDFF:	qweDIZ2:BDCO}
give that initially the center of the domain is not necessarily burning
(since~$u(0,0)$ can be equal to~$\overline\Theta-\beta$, which is less than the ignition temperature~$\Theta$).

In this spirit, the conclusion obtained in~\eqref{CPDFF:DIZ2.c.bdary} states that there exists a finite interval of time~$(0,t_\star)$
during which the flame propagates from the boundary to cover the entire available region.

The specific estimate in~\eqref{CPDFF:DIZ2.dwafsger2}
also gives that the temperature growth is at least linear: in fact, this growth rate is in general not better
than linear, see footnote~\ref{SHARP-CPDFF:DIZ2.dwafsger2} on page~\pageref{SHARP-CPDFF:DIZ2.dwafsger2},
and this shows an interesting structural difference with respect to the exponentially fast
invasion of the fire obtained in equation~\eqref{CPDFF:DIZ2.d} of Theorem~\ref{SOPTH}
for fires ignited at the center of the domain.

Another interesting difference between Theorems~\ref{SOPTH} and~\ref{DBDARYEFF} is that, while fire propagation from the center of the domain requires the environmental diffusion coefficient to be sufficiently small in order to prevent the heat from dispersing throughout the habitat, the description of fire propagation from the boundary does not require this condition. In fact, when the boundary is maintained above the ignition temperature, environmental diffusion actually favors the propagation of the fire (as quantified in~\eqref{CASDF:d2},
and notice that the higher the environmental diffusion coefficient~$c$, the shorter the burning time~$t_\star$,
as made precise in~\eqref{CASDF:d3}).

\subsection{Traveling fire waves}
A concerning scenario in practical applications also arises when the temperature increases at a constant rate. This possibility is excluded by the following result:

\begin{theorem}[Absence of vertically translating solutions] \label{ABSE}
Assume that 
for every nonempty open set $U\subseteq\Omega$, there exists $x_U\in\Omega$ such that
\begin{equation}\label{NONDEKE5}
\int_U K(x_U,y) dy > 0.
\end{equation}
Then, there exists no burning solution of~\eqref{MAIN:EQ:pa3} of the form
\begin{equation}\label{AB:e} u(x,t)=v(x)-\beta t,\end{equation}
where~$v\in C(\overline\Omega)\cap C^2(\Omega)$ and~$\beta\in\R\setminus\{0\}$.
\end{theorem}

The case~$\beta=0$ in~\eqref{AB:e} corresponds to stationary solutions of~\eqref{MAIN:EQ:pa3}
and we will address specifically this class of solutions in
a forthcoming work.\medskip

In view of Theorem~\ref{ABSE}, a related (though somehow technically different) question
focuses on the possible existence of traveling waves of burning solutions.

In this case, we consider a global equation of the form
\begin{equation}\label{MAIN:EQ:pa3-TW}\begin{split}
\partial_t u(x,t)=c\partial^2_x u(x,t)+\int_{\R}\big(u(y,t)-\Theta\big)_+\,K(x-y)\,dy,\end{split}
\end{equation}
where~$c\in(0,+\infty)$ is the diffusion coefficient and~$\Theta\in\R$ the ignition temperature,
and the equation now describes the temperature~$u$ at position~$x\in\R$ and time~$t\in\R$.

In this setting, we assume a bound on the interaction kernel, namely that
\begin{equation}\label{LAMB}
K(r)\le\Lambda \chi_{(-R,R)}(r),
\end{equation}
for some~$\Lambda$, $R>0$, and we have the following result:

\begin{theorem}[Existence of traveling waves for kernels with short-range interactions]\label{TRAVE}
Assume~\eqref{LAMB}.
For every~$\omega>0$ and~$\kappa>0$ there exists~$R_\star>0$, depending only
on~$c$, $\omega$ and~$\Lambda$, such that if~$R\in(0,R_\star)$ then there exists a solution of~\eqref{MAIN:EQ:pa3-TW} of the form
$$ u(x,t)=v(x+\omega t)+\Theta$$
for some~$v:\R\to\R$ satisfying~$v(0)=0$ and~$v'(0)=\kappa$.
\end{theorem}

Regarding this statement, we observe that the conditions~$v(0)=0$ and~$v'(0)=\kappa>0$
entail that~$u$ is a nontrivial, burning solution according to
the setting introduced on page~\pageref{BURSO}.

A variant of Theorem~\ref{TRAVE} consists in replacing the assumption that
the range~$R$ of interaction is sufficiently small with the one that the intensity~$\Lambda$
of the interaction kernel is sufficiently small, according to the following result:

\begin{theorem}[Existence of traveling waves for kernels with mild interactions]\label{TRAVE.BIS}
Assume~\eqref{LAMB}.
For every~$\omega>0$ and~$\kappa>0$ there exists~$\Lambda_\star>0$, depending only
on~$c$, $\omega$ and~$R$, such that if~$\Lambda\in(0,\Lambda_\star)$ then there exists a solution of~\eqref{MAIN:EQ:pa3-TW} of the form
$$ u(x,t)=v(x+\omega t)+\Theta$$
for some~$v:\R\to\R$ satisfying~$v(0)=0$ and~$v'(0)=\kappa$.\end{theorem}

{F}rom our perspective, the analysis of traveling waves is instrumental to showcase some dynamics of fire propagation over time and space, describing the steady-state progression of a fire front through a landscape. 
The description of fires traveling at a more or less constant spread rate is indeed
an established feature in the experimental literature related to bushfires, since
there is empirical
evidence that a fire ``appears to reach a
quasi-steady speed after some time'' (see~\cite[page~239 and Figure~6]{CHENEYGOULD95}).

In this sense, on the one hand, the mathematical setting needed in the analysis of traveling waves is just an idealization of the real world, since the analytical
framework relies on an infinite spatial domain, while landscapes are finite, and assumes uniform conditions, while wind and fuel availability typically vary across the landscape.
On the other hand, the analysis of traveling waves may come in handy when
focusing on localized fire dynamics, to capture the internal feature of fire spread: this is particularly realistic when the region of interest given by the fire front and its immediate surroundings is relatively small compared to the entire landscape. As a byproduct, the simplification of dealing with homogeneous infinite domains makes it easier to construct and investigate analytically steady shapes moving at constant speed, providing a simplified yet robust way to describe the spread rate of a fire.\medskip

In practice, these idealized one-dimensional fronts moving at a constant speed can describe concrete situations in which the fire propagation happens to be essentially transverse to the front and linear in time: see e.g. Figures~8 and~13 in~\cite{MORVAN2011469}.

One-dimensional moving fronts have also been studied in bushfire models to account for
the combined effects of wind and slope inclination (see e.g. equation~(1) in~\cite{56:BYUHBCSP},
and notice that, for a constant wind and slope, this equation prescribes a constant velocity of
the fire advance).

While the environmental features are not part of the analysis developed here, since the moving fronts in this paper are the outcome of the temperature variation created by the bushfire itself, constant speed lines can also be due to wind effects, see e.g. Figures~1 and~2 in~\cite{Cruz2019}.
Parallel front lines can also be produced by the specific structure of the territory, see Figure~\ref{figcond3},
and they occur very often in controlled burning, see e.g.
minutes 2:36--4:45
and 5:10--5:15 of the video \url{https://www.youtube.com/watch?v=inKxlK8OXG0}.
\medskip

\begin{figure}[h]
    \centering
    \includegraphics[height=0.3\textwidth]{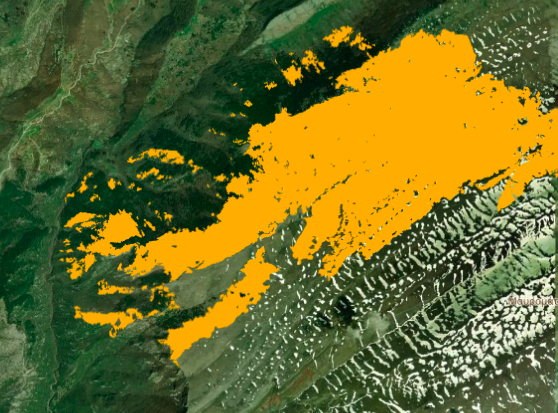}
    $\quad$
    \includegraphics[height=0.3\textwidth]{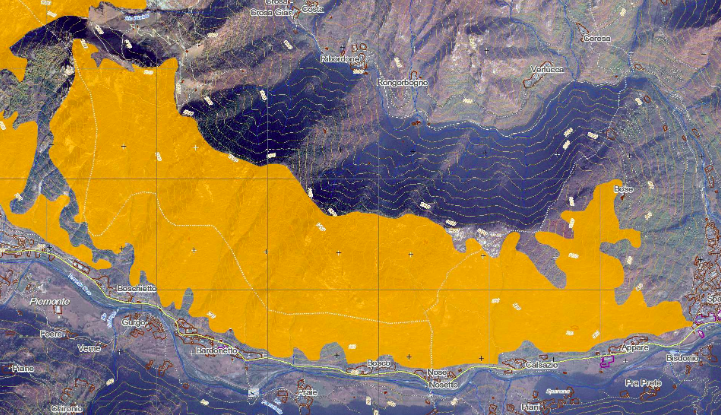}
    \caption{Delineation maps of actual bushfire events.
    Left: {\em EMSR747} Wildfire in Central Macedonia region, Greece, August 17, 2024.
Right: {\em EMSR253} Forest fire in Piemonte, Italy, October 27, 2017.
Images from the Copernicus Emergency Management Service, \url{https://rapidmapping.emergency.copernicus.eu/EMSR747},
\url{https://emergency.copernicus.eu/mapping/list-of-components/EMSR253}.}
    \label{figcond3}
\end{figure}

We observe that the traveling waves provided in the proof of Theorems~\ref{TRAVE}
and~\ref{TRAVE.BIS}
happen to be unbounded.
This is indeed a general feature, since bounded traveling waves do not exist:

\begin{theorem}[Absence of bounded traveling waves]\label{ABBOU}
Let~$\omega>0$ and~$K\in L^1(\R)$.
Assume that there exists a bounded solution of~\eqref{MAIN:EQ:pa3-TW} of the form
$$ u(x,t)=v(x+\omega t)+\Theta$$
for some~$v:\R\to\R$.

Then, $u$ is necessarily constant.
\end{theorem}

We think that Theorem~\ref{ABBOU} is interesting also because
it somehow explains the reason for which in Theorems~\ref{TRAVE} and~\ref{TRAVE.BIS}
one can construct traveling fronts of arbitrarily large speed~$\omega$.
This feature may seem, at a first glance, in contradiction with our practical
experience, since, typically, we expect that the maximal speed of propagation
of a fire line is dictated by environmental conditions such as fuel composition
and wind. 
In this spirit, Theorem~\ref{ABBOU} provides an explanation, given that the
drive for these ``fast spreading'' fronts comes, roughly speaking, from the
extremely high gradient temperature between the burning and unburned territories.
\medskip

We stress that,
as it can be readily checked, in the absence of the interaction kernel, traveling waves of the heat equation are monotone and of exponential type, namely of the form~$v(x)=\frac{\kappa(e^{\omega x}-1)}\omega$,
and the construction of the traveling waves in Theorems~\ref{TRAVE} and~\ref{TRAVE.BIS}
may look, at a first glance, a ``perturbative argument'' based on kernels with either short range of interaction
or small intensity. But the scenario is indeed more complex than that: indeed, the fact that the domain under consideration
is the whole real line (rather than a bounded set) may produce unexpected patterns, and this can be amplified by
the fact that, being the solution unbounded (as detected by Theorem~\ref{ABBOU}), divergent effects may influence the structure
of the global picture. In particular, the growth at infinity of the traveling waves is not the same as in the case of the heat equation,
as showcased in the next result:

\begin{theorem}[Exponential bounds for traveling waves]\label{NONMSPT43}
Let~$c=1$ and~$\omega>0$, and assume that there exists a solution of~\eqref{MAIN:EQ:pa3-TW} of the form
$$ u(x,t)=v(x+\omega t)+\Theta$$
for some~$v:\R\to\R$, with~$v(0)=0$ and~$v'(0)=\kappa\in(0,+\infty)$.

Then,
\begin{equation}\label{80iujgo486nb7m2985vb65BSAnmgyhn9}
\begin{split}
&{\mbox{$v'(x)\ge \kappa e^{\omega x}$ for all~$x\in(-\infty,0)$}}\\&
{\mbox{and $v'(x)\le \kappa e^{\omega x}$ for all~$x\in[0,+\infty)$.}}\end{split}
\end{equation}

However, if there exist~$\lambda$, $\varrho>0$ such that, for all~$r\in\R$,
\begin{equation}\label{LAMBdasotto}
K(r)\ge\lambda \chi_{(-\varrho,\varrho)}(r),
\end{equation}
then there cannot exist~$\kappa_\star\in(0,+\infty)$ such that, for all~$x\in[0,+\infty)$,
\begin{equation}\label{LAMBdasotto.bi}
v'(x)\ge \kappa_\star e^{\omega x}.
\end{equation}
\end{theorem}

Under additional assumptions, Theorem~\ref{ABBOU} can be sharpened
by detecting the side of the divergent structure of the traveling wave:

\begin{theorem}[Divergence of traveling waves at~$+\infty$]\label{ABBOU:side}
Let~$\omega>0$ and suppose that~$K$ satisfies~\eqref{LAMB}.
Assume that there exists a solution of~\eqref{MAIN:EQ:pa3-TW} of the form
$$ u(x,t)=v(x+\omega t)+\Theta$$
for some~$v:\R\to\R$, with~$v(0)=0$ and~$v'(0)=\kappa\in(0,+\infty)$.

Then, 
\begin{equation}\label{SIDE-la1}
\lim_{x\to-\infty}v(x) {\mbox{ exists, is finite, and nonpositive.}}\end{equation}

Also, 
\begin{equation}\label{SIDE-la2}{\limsup_{x\to+\infty}}|v(x)|=+\infty.\end{equation}
\end{theorem}

\subsection{Monotonicity issues}
Another interesting, and quite surprising, feature of these traveling waves is that,
differently from the case of the heat equation, they are not necessarily monotone, as described in the following result:

\begin{theorem}[Lack of monotonicity for traveling waves]\label{NONMSPT43.rm}
Assume~\eqref{LAMB} and~\eqref{LAMBdasotto}.
Let~$\omega$, $\kappa\in(0,+\infty)$ and assume that there exists a solution of~\eqref{MAIN:EQ:pa3-TW} of the form~$ v_\omega(x+\omega t)+\Theta$
for some~$v_\omega:\R\to\R$, with~$v_\omega(0)=0$ and~$v'_\omega(0)=\kappa$.

Then, given~$\omega_0>0$,
there exists~$\omega\in(0,\omega_0)$ such that the function~$v_\omega$ is not monotone nondecreasing.
\end{theorem}

The next result shows that traveling waves are always monotone in an interval that extends indefinitely to the left,
with an explicit quantification of the interval endpoint.

\begin{theorem}[Monotonicity in large intervals]\label{qdwsfvbolgrb:SDVb}
Let~$\omega>0$ and assume that~\eqref{LAMB} holds true. Consider a solution of~\eqref{MAIN:EQ:pa3-TW} of the form
$$ u(x,t)=v(x+\omega t)+\Theta$$
for some~$v:\R\to\R$, with~$v(0)=0$ and~$v'(0)=\kappa\in(0,+\infty)$.

Then, there exists~$L>0$,
depending only on~$\omega$, $c$, $\Lambda$, and~$R$, such that~$v'(x)>0$ for all~$x\in(-\infty,L]$.

Furthermore, when~$c=1$, $\omega=1$, and~$\Lambda\left(
e^R(e^R-1)-R\right)<1$, one can take
\begin{equation}\label{LESYT:Sdwoed}
L:=\ln\left(\frac{1+\Lambda(e^R+R-1)}{\Lambda(e^R-e^{-R})}\right).\end{equation}
\end{theorem}

The quantitative expression in~\eqref{LESYT:Sdwoed} is interesting, since its right-hand side diverges when either~$\Lambda\searrow0$ or~$R\searrow0$. This implies that when the interaction kernel is either of low intensity or short range, there exist traveling waves that are monotone over a very large interval extending indefinitely to the left. In fact, the interval can be arbitrarily large, provided that the intensity or range of the kernel is sufficiently small.

\subsection{Self-similar fire fronts}
Another question related to special solutions of moving fronts
deal with ``self-similar'' solutions. For instance, in the absence of
the interaction kernel~$K$, the classical heat equation~$\partial_tu=c\partial_x^2u$ exhibits solutions of the form 
\begin{equation}\label{PLOTFORFI2}\begin{split}&
u(x,t)=v\left( \lambda(t)\,x\right),\\&{\mbox{with}}\quad 
v(r)=v_0+v_1\int_0^r e^{-\frac{a\varrho^2}{2}}\,d\varrho
\quad{\mbox{and}}\quad 
\lambda(t)=\frac{\lambda_0}{\sqrt{1+2ac\lambda_0^2 t}}
,\end{split}\end{equation}
with~$\lambda_0$, $v_0$, $v_1$ and~$a\in\R$.

These solutions are defined for all~$x\in\R$ and, provided
that~$1+2ac\lambda_0^2 t>0$,
they are the only ones that possess the remarkable feature of having the ``same shape'' for all times,
up to a spatial rescaling, see Figure~\ref{fighss}.

\begin{figure}[t]
    \centering
    \includegraphics[width=0.45\textwidth]{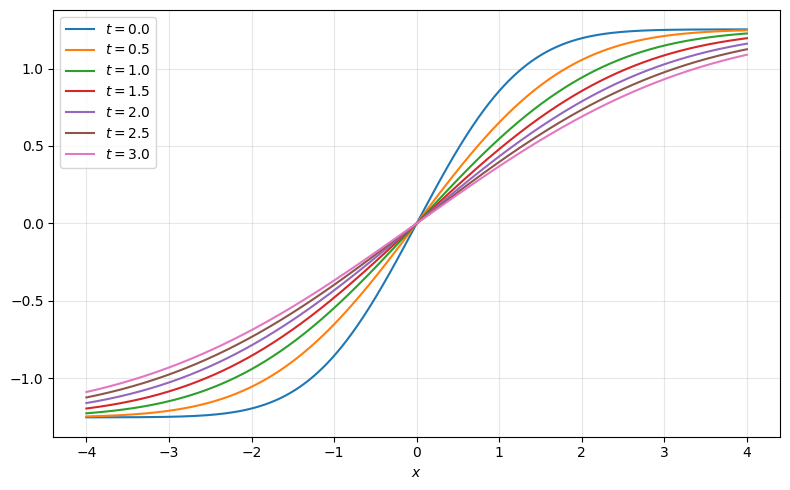}
        \includegraphics[width=0.45\textwidth]{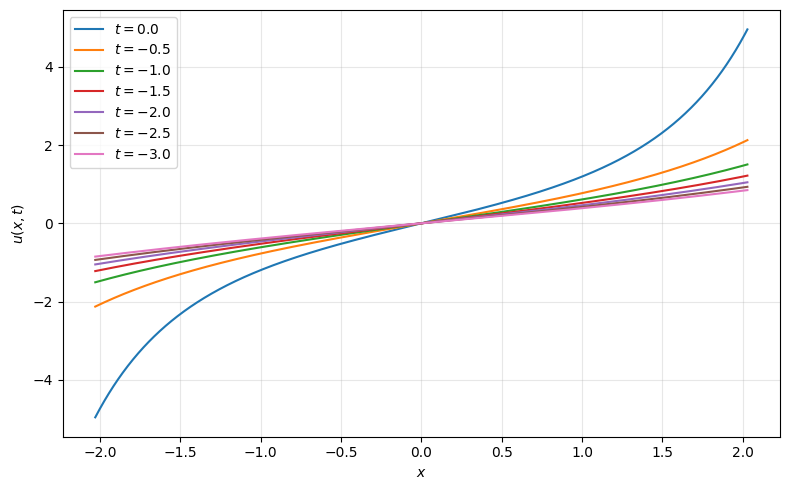}
    \caption{{Plot of the function in~\eqref{PLOTFORFI2}
    with~$c:=1$, $\lambda_0:=1$, $v_0:=0$ and~$v_1:=1$.
    Left: case~$a:=1$ and~$t\in\{0, \, 0.5,\, 1, \, 1.5,\,2,\,2.5,\,3\}$.
    Right: case~$a:=-1$ and~$t\in\{-3,\,-2.5,\,-2,\,-1.5,\,-1,\,-0.5,\,0\}$.}
    }
    \label{fighss}
\end{figure}

In the presence of an interaction kernel, in general
we cannot expect self-similar solutions of the bushfire equation (other than
the trivial ones coming from the heat equation):

\begin{theorem}[Absence of self-similar solutions]\label{sojdcmvvb-rw-SeeDeqwdf}
Let~$\ell>0$.
Let~$r_0\in\R$, $v\in C^2(\R)$ and suppose that~$v\le\Theta$ in~$(-\infty,r_0]$.

Let also~$T>0$ and~$\lambda\in C^1([0,T],(0,+\infty))$.
Suppose that either~$\lambda$ vanishes identically or it is a
non-constant function.

Suppose that~$u$ is a solution of 
\begin{equation}\label{PASD:erf} \partial_t u(x,t)=c\Delta u(x,t)+\int_{x-\ell}^{x+\ell}\big(u(y,t)-\Theta\big)_+\,dy,\end{equation}
for all~$x\in\R$ and~$t\in(0,T)$, having the form
\begin{equation}\label{PASD:erf17h} u(x,t)=v\left( \lambda(t)\,x\right).\end{equation}

Then, $u(x,t)\le\Theta$ for all~$x\in\R$ and~$t\in[0,T]$,
and the representation formula in~\eqref{PLOTFORFI2} holds true.
\end{theorem}

We stress that the classification in Theorem~\ref{sojdcmvvb-rw-SeeDeqwdf}
reduces the problem to a non-burning solution, according to the terminology
of page~\pageref{BURSO}. Namely, Theorem~\ref{sojdcmvvb-rw-SeeDeqwdf}
states that self-similar solutions of~\eqref{PASD:erf} always remain
below the ignition temperature and therefore
they merely reduce to the diffusive solutions of the heat equation
(hence, the bushfire equation~\eqref{PASD:erf}
does not possess burning solutions of self-similar type).

\subsection{Organization of the paper}
The rest of this paper is devoted to the proofs of the above results.
Specifically, Lemma~\ref{COMPAPI} is proved in Section~\ref{SEC:2},
and it is then used for the proof of 
Lemma~\ref{NECE}, as well as those of Theorems~\ref{SOPTH}, \ref{SOPTH:2}, and~\ref{DBDARYEFF},
which are contained, respectively, in Sections~\ref{SEC:3}, \ref{SOPTH:sec}, \ref{SOPTH:sec:2}, and~\ref{DBDARYEFF:SE}.

Also, 
Theorem~\ref{ABSE} is proved in Section~\ref{SEC4}
and Section~\ref{SPJD0tkjy5pu:0345} contains some useful auxiliary observations
allowing one to conveniently rewrite the equation of traveling waves.

With this,
Theorem~\ref{TRAVE} is established in Section~\ref{DS5} and Theorem~\ref{TRAVE.BIS}
in Section~\ref{DS5.bBSIS}.

The proof of Theorem~\ref{ABBOU} is contained in Section~\ref{ajopsclnFG}, 
that of Theorem~\ref{NONMSPT43} in Section~\ref{N56789-97ONMSPT43:OSIDNKF3495TI:2345T0},
that of
Theorem~\ref{ABBOU:side} in Section~\ref{oqwrufjgr0tphknwfbn},
that of Theorem~\ref{NONMSPT43.rm} in Section~\ref{2efm02emiaksMEQWfvqw01367vf2-323r},
that of Theorem~\ref{qdwsfvbolgrb:SDVb} in Section~\ref{qdwsfvbolgrb:SDVb:SEC}.

The proof of Theorem~\ref{sojdcmvvb-rw-SeeDeqwdf}
is presented in Section~\ref{sojdcmvvb-rw-SeeDeqwdf:S}.

\section{Proof of Lemma~\ref{COMPAPI}}\label{SEC:2}
Recalling assumption~\eqref{UNIKe}, we let
$$ A:=1+2\sup_{x\in\Omega}\int_{\Omega}K(x,y)\,dy.$$
We claim that
\begin{equation}\label{RIDU}
{\mbox{$u(x,t)\le v(x,t)$ for all~$x\in\Omega$ and~$t\in\displaystyle\left[0,\frac1{ A }\right]$.}}\end{equation}
To prove this, suppose, for the sake of contradiction, that the claim is not true and there exist~$x_\star\in\Omega$ and~$t_\star\in\left[0,\frac1{ A }\right]$ such that~$v(x_\star,t_\star)<u(x_\star,t_\star)$. 

We pick~$\delta>0$ sufficiently small such that
$$\delta+\delta A t_\star< u(x_\star,t_\star)-v(x_\star,t_\star) $$
and let
$$ W_\delta(x,t):=v(x,t)-u(x,t)+\delta+\delta At.$$ We remark that, for all~$x\in\Omega$,
$$ W_\delta(x,0)\ge\delta>0$$
and
$$ W_\delta(x_\star,t_\star)= v(x_\star,t_\star)-u(x_\star,t_\star)+\delta+\delta A t_\star<0.$$
By continuity, we have that~$W_\delta<0$ in a neighborhood of~$(x_\star,t_\star)$.
Hence there exists~$\tau_\delta\in(0,t_\star)$ such that~$W_\delta(x,t)\ge0$ for all~$x\in\Omega$ and~$t\in[0,\tau_\delta]$
and there exist an infinitesimal sequence~$\e_j\searrow0$ as~$j\to+\infty$ and points~$\tilde x_j\in\overline\Omega$ for which~$W_\delta(\tilde x_j,\tau_\delta+\e_j)<0$.

In particular, if~$x_j\in\overline\Omega$ is such that
\begin{equation}\label{KSmi} W_\delta(x_j,\tau_\delta+\e_j)=\min_{x\in\overline\Omega}W_\delta(x,\tau_\delta+\e_j),\end{equation}
we have that
$$ W_\delta(x_j,\tau_\delta+\e_j)\le W_\delta(\tilde x_j,\tau_\delta+\e_j)<0.$$

Up to a subsequence, we can suppose that~$x_j\to \eta_\delta$ for some~$\eta_\delta\in\overline\Omega$.
Moreover,
\begin{eqnarray*}&& v(\eta_\delta,\tau_\delta)-u(\eta_\delta,\tau_\delta)=\lim_{j\to+\infty} v(x_j,\tau_\delta+\e_j)-u(x_j,\tau_\delta+\e_j)\\&&\qquad=
\lim_{j\to+\infty} W_\delta(x_j,\tau_\delta+\e_j)-\delta-\delta A\tau_\delta\le
-\delta<0\end{eqnarray*}
and therefore~$\eta_\delta\in\Omega$.

On this account, we have that~$x_j\in\Omega$ provided that~$j$ is sufficiently large and therefore, by~\eqref{KSmi}, $x_j$ is an interior minimum for the function~$\Omega\ni x\mapsto W_\delta(x,\tau_\delta+\e_j)$, yielding that
$$ \Delta W_\delta(x_j,\tau_\delta+\e_j)\ge0,$$
and therefore
$$ \Delta W_\delta (\eta_\delta, \tau_\delta)\ge0.$$

Furthermore,
$$ -\partial_t W_\delta(\eta_\delta,\tau_\delta)=\lim_{j\to+\infty}\frac{W_\delta(\eta_\delta,\tau_\delta-\e_j)-W_\delta(\eta_\delta,\tau_\delta)}{\e_j}
=\lim_{j\to+\infty}\frac{W_\delta(\eta_\delta,\tau_\delta-\e_j)}{\e_j}
\ge0
$$
and consequently, by~\eqref{MAIN:EQ:pa3:sSop},
\begin{equation*}
\begin{split}
0&\ge\partial_t W_\delta(\eta_\delta,\tau_\delta)-c\Delta W_\delta(\eta_\delta,\tau_\delta)
\\&=\partial_t v(\eta_\delta,\tau_\delta)-\partial_t u(\eta_\delta,\tau_\delta)+\delta A
-c\Delta v(\eta_\delta,\tau_\delta)+c\Delta u(\eta_\delta,\tau_\delta)
\\&\ge
\int_{\Omega}\big(v(y,\tau_\delta)-\Theta\big)_+\,K(\eta_\delta,y)\,dy-\int_{\Omega}\big(u(y,\tau_\delta)-\Theta\big)_+\,K(\eta_\delta,y)\,dy+\delta A
\\&=\int_{\Omega}\big(W_\delta(y,\tau_\delta)+
u(y,\tau_\delta)-\delta-\delta A\tau_\delta
-\Theta\big)_+\,K(\eta_\delta,y)\,dy\\&\qquad\qquad-\int_{\Omega}\big(u(y,\tau_\delta)-\Theta\big)_+\,K(\eta_\delta,y)\,dy+\delta A\\&\ge\int_{\Omega}\big(u(y,\tau_\delta)-\delta-\delta A\tau_\delta
-\Theta\big)_+\,K(\eta_\delta,y)\,dy\\&\qquad\qquad-\int_{\Omega}\big(u(y,\tau_\delta)-\Theta\big)_+\,K(\eta_\delta,y)\,dy+\delta A\\&
\ge \big(-\delta-\delta A\tau_\delta\big)\int_{\Omega}K(\eta_\delta,y)\,dy+\delta A.
\end{split}
\end{equation*}
Dividing by~$\delta$ we thereby find that
\begin{eqnarray*}&& 1+2\sup_{x\in\Omega}\int_{\Omega}K(x,y)\,dy=
A\le (1+A\tau_\delta) \sup_{x\in\Omega}\int_{\Omega}K(x,y)\,dy
\\&&\qquad\le (1+A t_\star) \sup_{x\in\Omega}\int_{\Omega}K(x,y)\,dy
\le2\sup_{x\in\Omega}\int_{\Omega}K(x,y)\,dy,
\end{eqnarray*}
which is a contradiction.
With this, the claim in~\eqref{RIDU} is established.

Now we claim that, for every~$m\in\N\setminus\{0\}$,
\begin{equation}\label{RIDU2}
{\mbox{$u(x,t)\le v(x,t)$ for all~$x\in\Omega$ and~$t\in\displaystyle\left[0,\frac{m}{ A }\right]$.}}\end{equation}
For this, we can argue by induction.
Indeed, when~$m=1$ the claim in~\eqref{RIDU2} follows from~\eqref{RIDU}.

Suppose now that the claim in~\eqref{RIDU2} is valid for some~$m$ and let us prove it for~$m+1$. To this end, we let~$\widetilde u(x,t):=u\left(x,t+\frac{m}{ A }\right)$ and~$\widetilde v(x,t):=v\left(x,t+\frac{m}{ A }\right)$, we observe that~$\widetilde u$ and~$\widetilde v$ are also
as in~\eqref{MAIN:EQ:pa3:sSop} with~$\widetilde u(x,0)=u\left(x,\frac{m}{ A }\right)\le v\left(x,\frac{m}{ A }\right)=\widetilde v(x,0)$ for all~$x\in\Omega$, thanks to the inductive assumption, and~$\widetilde u(x,t)\le\widetilde v(x,t)$ for all~$x\in\partial\Omega$ and~$t\in[0,+\infty)$, thanks to the assumptions in Lemma~\ref{COMPAPI}. 

Hence, we can apply~\eqref{RIDU} to~$\widetilde u$ and~$\widetilde v$, concluding that, for all~$x\in\Omega$ and~$t\in\displaystyle\left[0,\frac1{ A }\right]$,
$$ u\left(x,t+\frac{m}{ A }\right)=\widetilde u(x,t)\le\widetilde v(x,t)=v\left(x,t+\frac{m}{ A }\right),$$
from which~\eqref{RIDU2} follows.

The claim in Lemma~\ref{COMPAPI} is now a consequence of~\eqref{RIDU2}.~\hfill \qedsymbol

\section{Proof of Lemma~\ref{NECE}}\label{SEC:3}
We observe that~$v(x,t):=\Theta$ is a solution of~\eqref{MAIN:EQ:pa3} such that~$u(x,0)\le\Theta= v(x,0)$
for all~$x\in\Omega$ and~$u(x,t)\le\Theta=v(x,t)$ for all~$x\in\partial\Omega$ and~$t\in[0,+\infty)$. The desired result then follows from Lemma~\ref{COMPAPI}.~\hfill \qedsymbol

\section{Proof of Theorem~\ref{SOPTH}}\label{SOPTH:sec}
Let
$$ \underline{u}(x,t):= \lambda_0 e^{\alpha t} (1-|x|^2),$$
with~$\alpha>0$ for us to choose conveniently small in what follows.

Notice that, if~$x\in\partial B_1$ and~$t\in[0,+\infty)$,
\begin{equation}\label{CPDFF:DIZ3}
\underline{u}(x,t)=0=u(x,t).
\end{equation}
In addition, by~\eqref{CPDFF:DIZ2}, for all~$x\in B_1$,
\begin{equation}\label{CPDFF:DIZ4}
\underline{u}(x,0)=\lambda_0 (1-|x|^2)
\le u(x,0).
\end{equation}

We also have that, for all~$x\in B_1$ and~$t\in(0,+\infty)$,
\begin{equation}\label{PSDSDVFGRGFBEFSDV}
\partial_t \underline u(x,t)-c\Delta\underline u(x,t)=
\lambda_0 e^{\alpha t}
\big( \alpha (1-|x|^2)+2nc \big).
\end{equation}

In addition, given~$\tau\in(0,1)$, for all~$y\in B_{\tau}$,
\begin{eqnarray*}&&
\underline u(y,t)-\Theta\ge \lambda_0 e^{\alpha t} (1-\tau^2)
-\Theta=(\lambda_0-\Theta) e^{\alpha t} (1-\tau^2)
+\Theta \big(e^{\alpha t} (1-\tau^2)-1\big)\\&&\qquad\ge
(\lambda_0-\Theta) e^{\alpha t} (1-\tau^2)
+\Theta \big( (1-\tau^2)-1\big)
= (\lambda_0-\Theta) e^{\alpha t} (1-\tau^2)
-\Theta \tau^2,
\end{eqnarray*}
which is nonnegative as long as~$(\lambda_0-\Theta) (1-\tau^2)
\ge\Theta \tau^2$, and 
this is warranted if we choose
$$\tau:= \sqrt{\frac{\lambda_0 - \Theta}{2\lambda_0}}.$$

In this way, we have found that
\begin{equation*}
\begin{split}&
\int_{B_1}\big(\underline u(y,t)-\Theta\big)_+\,K(x,y)\,dy\ge
\int_{B_\tau}\big(\underline u(y,t)-\Theta\big)_+\,K(x,y)\,dy\\&\qquad\ge
\int_{B_\tau}
\Big( (\lambda_0-\Theta) e^{\alpha t} (1-\tau^2)
-\Theta \tau^2\Big)
\,K(x,y)\,dy\\&\qquad=
\int_{B_{\sqrt{\frac{\lambda_0 - \Theta}{2\lambda_0}}}}\left(
\frac{(\lambda_0-\Theta)\Theta}{2\lambda_0} \big(e^{\alpha t}-1\big)+\frac{(\lambda_0-\Theta)e^{\alpha t}}{2}\right)
\,K(x,y)\,dy.
\end{split}
\end{equation*}
Hence, by~\eqref{CPDFF:DIZ0},
\begin{equation*}
\int_{B_1}\big(\underline u(y,t)-\Theta\big)_+\,K(x,y)\,dy\ge
c_0 \big(e^{\alpha t}-1\big)+c_1 e^{\alpha t},
\end{equation*}
for some~$c_0$, $c_1>0$ depending only on~$\lambda_0 $, $\Theta$, and~$K$.

By combining this information and~\eqref{PSDSDVFGRGFBEFSDV} it follows that
\begin{eqnarray*}
&&\partial_t \underline u(x,t)-c\Delta\underline u(x,t)-
\int_{B_1}\big(\underline u(y,t)-\Theta\big)_+\,K(x,y)\,dy\\
&&\qquad\le
\lambda_0 e^{\alpha t}
\big( \alpha (1-|x|^2)+2nc \big)
-c_0 \big(e^{\alpha t}-1\big)-c_1 e^{\alpha t}\\
&&\qquad=
\Big(\lambda_0 \big( \alpha (1-|x|^2) +2nc\big)-c_0-c_1\Big) e^{\alpha t}
+c_0 \\&&\qquad\le
\Big(\lambda_0 (\alpha+2nc)-c_0-c_1\Big) e^{\alpha t}+c_0 .
\end{eqnarray*}

Now, if~$\alpha$ is chosen conveniently small
(depending only on~$\lambda_0 $, $\Theta$, and~$K$), we can suppose that~$\alpha\le\frac{c_1}{4\lambda_0}$. Also, if~$c$ is 
chosen conveniently small
(depending only on~$n$, $\lambda_0 $, $\Theta$, and~$K$), we can suppose
that~$c\le\frac{c_1}{8n\lambda_0}$.
With these choices, we have that~$
\lambda_0 (\alpha+2nc)-c_0-c_1<0$, and therefore
\begin{eqnarray*}
&&\partial_t \underline u(x,t)-c\Delta\underline u(x,t)-
\int_{B_1}\big(\underline u(y,t)-\Theta\big)_+\,K(x,y)\,dy\\
&&\qquad\le
\Big(\lambda_0 (\alpha+2nc)-c_0-c_1\Big) +c_0 \\&&\qquad=
\lambda_0 (\alpha+2nc)-c_1\\&&\qquad=
-\frac{c_1}2 .
\end{eqnarray*}
Thanks to this inequality, \eqref{CPDFF:DIZ3}, and~\eqref{CPDFF:DIZ4},
we can employ the Comparison Principle in Lemma~\ref{COMPAPI} and conclude that
$$ u(x,t)\ge\underline u(x,t)=\lambda_0 e^{\alpha t} (1-|x|^2),$$
as desired.~\hfill \qedsymbol

\section{Proof of Theorem~\ref{SOPTH:2}}\label{SOPTH:sec:2}
Let~$\alpha:=2nc-C>0$, thanks to~\eqref{CPDFF:DIZ2.b:2}, and define
$$ \overline{u}(x,t):= \lambda_0 e^{-\alpha t} (1-|x|^2).$$
We point out that that, if~$x\in\partial B_1$ and~$t\in[0,+\infty)$,
\begin{equation}\label{CPDFF:DIZ3:2}
\overline{u}(x,t)=0=u(x,t)
\end{equation}
and, by means of~\eqref{CPDFF:DIZ2:2}, for all~$x\in B_1$,
\begin{equation}\label{CPDFF:DIZ4:2}
\overline{u}(x,0)=\lambda_0 (1-|x|^2)
\ge u(x,0).
\end{equation}

Besides, for all~$x\in B_1$ and~$t\in(0,+\infty)$,
\begin{equation}\label{odnfwepoegihpwerhkngb}
\partial_t \overline u(x,t)-c\Delta\overline u(x,t)=
\lambda_0 e^{-\alpha t}
\big( 2nc -\alpha (1-|x|^2)\big).
\end{equation}

Furthermore, by~\eqref{CPDFF:DIZ} and the monotonicity of the function~$\R\ni r\mapsto r_+$,
we have that
\begin{eqnarray*}
&&\int_{B_1}\big(\overline u(y,t)-\Theta\big)_+\,K(x,y)\,dy\le
\int_{B_1}\overline u_+(y,t)\,K(x,y)\,dy\\&&\qquad
=\lambda_0 e^{-\alpha t} \int_{B_1}(1-|y|^2)\,K(x,y)\,dy\le C
\lambda_0 e^{-\alpha t},
\end{eqnarray*}
where~\eqref{ihsdcknvAS234MSsijsmdvlpsd} has been used in the latter inequality.

Hence, recalling~\eqref{odnfwepoegihpwerhkngb}, \begin{eqnarray*}&&
\partial_t \overline u(x,t)-c\Delta\overline u(x,t)-\int_{B_1}\big(\overline u(y,t)-\Theta\big)_+\,K(x,y)\,dy\\&&\qquad\ge
\lambda_0 e^{-\alpha t}
\big( 2nc -\alpha (1-|x|^2)-C\big)\\&&\qquad\ge
\lambda_0 e^{-\alpha t}
\big( 2nc -\alpha-C\big)=0.
\end{eqnarray*}
This, \eqref{CPDFF:DIZ3:2}, and~\eqref{CPDFF:DIZ4:2},
combined with the Comparison Principle in Lemma~\ref{COMPAPI}, entail that
$$ u(x,t)\le\overline u(x,t)=\lambda_0 e^{-\alpha t} (1-|x|^2),$$
as desired.~\hfill \qedsymbol

\section{Proof of Theorem~\ref{DBDARYEFF}}\label{DBDARYEFF:SE}
We define
$$\underline u(x,t):=\overline\Theta-(\beta-\alpha t)(1-|x|^2),$$
with~$\alpha>0$ as in~\eqref{CASDF:d2}.

We notice that, for all~$x\in B_1$,
\begin{equation}\label{sdacdCSXvSihfDbk-90ieRrRjf-01}
\underline u(x,0)=\overline\Theta-\beta(1-|x|^2)\le u(x,0),\end{equation}  thanks to~\eqref{CPDFF:	qweDIZ2:BDCO}.

Besides, for all~$x\in\partial B_1$ and~$t\in[0,+\infty)$, 
\begin{equation}\label{sdacdCSXvSihfDbk-90ieRrRjf-02}
\underline u(x,t)=\overline\Theta= u(x,t).\end{equation}

Furthermore,
\begin{equation*}
\partial_t\underline u(x,t)-c\Delta\underline u(x,t)=
\alpha (1-|x|^2)-2nc(\beta-\alpha t).
\end{equation*}
We also remark that, when~$t\in[0,t_\star]$, it holds that
\begin{equation*}
\beta-\alpha t\ge\overline\Theta-\Theta>0.
\end{equation*}
As a consequence,
\begin{eqnarray*}&&
\partial_t\underline u(x,t)-c\Delta\underline u(x,t)-\int_{B_1}\big(\underline u(y,t)-\Theta\big)_+\,K(x,y)\,dy\\&&\qquad\le
\alpha (1-|x|^2)-2nc(\beta-\alpha t)\\&&\qquad\le
\alpha-2nc(\overline\Theta-\Theta).
\end{eqnarray*}

This and~\eqref{CASDF:d2} yield that
$$\partial_t\underline u(x,t)-c\Delta\underline u(x,t)-\int_{B_1}\big(\underline u(y,t)-\Theta\big)_+\,K(x,y)\,dy\le0.$$
Thus, recalling~\eqref{sdacdCSXvSihfDbk-90ieRrRjf-01} and~\eqref{sdacdCSXvSihfDbk-90ieRrRjf-02},
we can utilize the Comparison Principle in Lemma~\ref{COMPAPI} and conclude that~$u(x,t)\ge\underline u(x,t)$
for all~$x\in B_1$ and~$t\in[0,t_\star]$. This establishes the claim in~\eqref{CPDFF:DIZ2.dwafsger2}, which\footnote{For completeness, we point out that if
$$\overline u(x,t):=\overline\Theta+at,$$
with~$a>0$ and~$t\in[0,1]$, if the interaction kernel is bounded by a small quantity~$\e$ we have that
\begin{equation*}
\begin{split}
&\int_{B_1}\big(\overline u(y,t)-\Theta\big)_+\,K(x,y)\,dy\le
C\e\,(\overline\Theta-\Theta+at)\le C\e\,(\overline\Theta-\Theta+a),
\end{split}
\end{equation*}for some~$C>0$, and thus
\begin{eqnarray*}&&
\partial_t\overline u(x,t)-c\Delta\overline u(x,t)-\int_{B_1}\big(\overline u(y,t)-\Theta\big)_+\,K(x,y)\,dy\\&&\qquad\ge
a-C\e\,(\overline\Theta-\Theta+a)\ge0,
\end{eqnarray*}
as long as~$a\ge\frac{C\e\,(\overline\Theta-\Theta+a)}{1-C\e}$.

In this situation, for a solution~$u$ with initial datum below~$\overline\Theta$,
the Comparison Principle in Lemma~\ref{COMPAPI} would have returned that~$u(x,t)\le \overline u(x,t)$
for all~$x\in B_1$ and~$t\in[0,1]$.\label{SHARP-CPDFF:DIZ2.dwafsger2}

This shows that, in general, the linear growth rate obtained in~\eqref{CPDFF:DIZ2.dwafsger2}
is sharp and cannot be improved.}
in turn implies~\eqref{CPDFF:DIZ2.c.bdary} as well.~\hfill \qedsymbol

\section{Proof of Theorem~\ref{ABSE}}\label{SEC4}
Suppose that there exists a solution of~\eqref{MAIN:EQ:pa3} in the form given by~\eqref{AB:e}.
Then, by~\eqref{MAIN:EQ:pa3}, for all~$x\in\Omega$ and~$t\in(0,+\infty)$,
\begin{equation}\label{PI02}
\begin{split}&-\beta=
\partial_t u(x,t)=c\Delta u(x,t)+\int_{\Omega}\big(u(y,t)-\Theta\big)_+\,K(x,y)\,dy
\\&\qquad\qquad= c\Delta v(x)+\int_{\Omega}\big(v(y)-\beta t-\Theta\big)_+\,K(x,y)\,dy.\end{split}\end{equation}

Now, if~$\beta>0$ we pick~$t\ge \frac{\|v\|_{L^\infty(\Omega)}-\Theta}{\beta}$ and deduce from~\eqref{PI02} that, for all~$x\in\Omega$,
$$ c\Delta v(x)=-\beta.$$
Plugging this information back into~\eqref{PI02}, we find that
\begin{equation}\label{AlT7i2TnoelKmde}\int_{\Omega}\big(v(y)-\beta t-\Theta\big)_+\,K(x,y)\,dy=0.\end{equation}

Now, we prove that
\begin{equation}\label{UNOTBUG}
{\mbox{$u$ is not a burning solution.}}
\end{equation}
To prove this, we argue for the sake of contradiction and suppose that
there exist~$p_\star\in\Omega$ and~$t_\star\in(0,+\infty)$ such that~$v(p_\star)-\beta t_\star>\Theta$.
Then, by continuity, we find~$b>0$ and a small ball~$B$ such that~$p\in B\subseteq\Omega$ and~$v(y)-\beta t_\star-\Theta\ge b$ for all~$y\in B$.

This and~\eqref{AlT7i2TnoelKmde} give that
$$0\ge\int_{B}\big(v(y)-\beta t-\Theta\big)_+\,K(x,y)\,dy\ge b\int_{B}K(x,y)\,dy.$$
Hence, owing to~\eqref{NONDEKE5}, there exists~$x_B\in\Omega$ such that
$$0\ge b\int_{B}K(x_B,y)\,dy>0,$$
which is a contradiction, and the proof of~\eqref{UNOTBUG} is complete.

As a result, to have a burning solution, necessarily~$\beta\le0$.
In this situation, we infer from~\eqref{PI02} that for all~$x\in\Omega$ and~$T>t>0$,
\begin{equation*}
\begin{split}0&=
-\beta-c\Delta v(x)+\beta+c\Delta v(x)\\& =
\int_{\Omega}\big(v(y)-\beta T-\Theta\big)_+\,K(x,y)\,dy
-\int_{\Omega}\big(v(y)-\beta t-\Theta\big)_+\,K(x,y)\,dy
\\&=
\int_{\Omega} \Big(\big(v(y)+|\beta| T-\Theta\big)_+ -\big(v(y)+|\beta| t-\Theta\big)_+\Big)\,K(x,y)\,dy
\end{split}\end{equation*}
and therefore, by the monotonicity of the integrand in the time variable,
\begin{equation*}
\big(v(y)+|\beta| T-\Theta\big)_+ =\big(v(y)+|\beta| t-\Theta\big)_+.
\end{equation*}

This entails that, for all~$y\in\Omega$ and~$T>0$,
\begin{equation*}
\big(v(y)+|\beta| T-\Theta\big)_+ =\big(v(y)-\Theta\big)_+.
\end{equation*}
This and the monotonicity involved give a contradiction unless~$\beta=0$, from which the claim in Theorem~\ref{ABSE} follows.~\hfill \qedsymbol

\section{Some auxiliary observations}\label{SPJD0tkjy5pu:0345}
In this section we rephrase the notion of traveling wave solution
in a form which is suitable for the proofs of Theorems~\ref{TRAVE} and~\ref{TRAVE.BIS}.
The idea is to combine integration and extension method to reduce the problem to
a fixed-point argument in a convenient (not standard) functional space.

\begin{lemma}\label{lSJD}
The following conditions are equivalent:
\begin{itemize}
\item The function
\begin{equation}\label{ChCSKJ.oqdjf34rg0hn.354} u(x,t)=v(x+\omega t)+\Theta,\end{equation}
with~$v:\R\to\R$,
is a solution of~\eqref{MAIN:EQ:pa3-TW},
\item $v$ is a solution of
\begin{equation}\label{ChCSKJ.oqdjf34rg0hn.35}\omega v'-cv''= v_+*K.\end{equation}
\end{itemize}
\end{lemma}

\begin{proof} On the one hand, we rewrite \eqref{MAIN:EQ:pa3-TW} in the form
\begin{eqnarray*}&&\omega
v'(x+\omega t)=
\partial_t u(x,t)=c\partial_x^2 u(x,t)+\int_{\R}\big(u(y,t)-\Theta\big)_+\,K(x-y)\,dy\\&&\qquad
=cv''(x+\omega t)+\int_{\R}v_+(y+\omega t)\,K(x-y)\,dy\\&&\qquad
=cv''(x+\omega t)+\int_{\R}v_+(Y)\,K(x+\omega t-Y)\,dY
,\end{eqnarray*}
and then
\begin{equation*}\omega
v'(x)=cv''(x)+\int_{\R}v_+(y)\,K(x-y)\,dy=cv''(x)+v_+*K(x),
\end{equation*}which gives~\eqref{ChCSKJ.oqdjf34rg0hn.35}.

On the other hand, if~$v$ solves~\eqref{ChCSKJ.oqdjf34rg0hn.35} and~$u$ is as in~\eqref{ChCSKJ.oqdjf34rg0hn.354}, then
\begin{eqnarray*}&&
\partial_tu(x,t)-c\partial_x^2 u(x,t)=\omega v'(x+\omega t)-cv''(x+\omega t)
=v_+*K(x+\omega t)\\&&\,=\int_\R v_+(y)\,K(x+\omega t-y)\,dy
=\int_\R v_+(Y+\omega t)\,K(x-Y)\,dY\\&&\;=
\int_\R \big(u(Y,t)-\Theta\big)_+\,K(x-Y)\,dY,
\end{eqnarray*}
which entails that~$u$ is a solution of~\eqref{MAIN:EQ:pa3-TW}.\end{proof}

For finite-range interaction kernels, it actually
suffices to solve~\eqref{ChCSKJ.oqdjf34rg0hn.35} in~$[-2R,+\infty)$,
since one can then proceed with an extension method. The precise result goes as follows:

\begin{lemma}\label{EXT:v}
Assume~\eqref{LAMB}. Suppose that~$v\in C^2([-2R,+\infty))$ is a solution of~\eqref{ChCSKJ.oqdjf34rg0hn.35} in~$[-2R,+\infty)$
with~$v(x)\le0\le v'(x)$ for all~$x\in[-2R,0]$.

Then, one can extend~$v$ to a solution of~\eqref{ChCSKJ.oqdjf34rg0hn.35} in the whole of~$\R$.
\end{lemma}

\begin{proof}Suppose that~$v\in C^2([-2R,+\infty))$ is as in the statement
of Lemma~\ref{EXT:v} and consider the following extension:
for all~$x\in(-\infty,-2R)$, let
\begin{equation}\label{EXT:v132r3we} v(x):=\frac{v'(-2R)\,e^{\omega (x+2R)}}\omega +v(-2R)-\frac{v'(-2R)}\omega.\end{equation}
Note that
$$ \lim_{x\searrow-2R}v(x)=\lim_{x\nearrow-2R}v(x)\quad{\mbox{ and }}\quad
\lim_{x\searrow-2R}v'(x)=\lim_{x\nearrow-2R}v'(x),$$
giving that\begin{equation}\label{EXT:v132r3.304}
v\in C^1(\R).\end{equation}

In addition, we have that~$\frac{v'(-2R)}\omega\ge0$ and~$v(-2R)\le0$, therefore
\begin{equation}\label{EXT:v132r3}
{\mbox{$v\le0$ and~$v'\ge0$ in~$(-\infty,0]$.}}\end{equation}

We also observe that, since~$v\le0$ in~$(-\infty,0]$, we have that, for all~$x<-R$,
\begin{equation*}\begin{split}& v_+*K(x)=
\int_{\{ y\in(-R,R)\cap(-\infty,x)\}} v_+(x-y)\,K(y)\,dy\\&\qquad=\int_{\{ y\in\varnothing\}} v_+(x-y)\,K(y)\,dy=0.
\end{split}
\end{equation*}

Hence, since~$v$ satisfies~\eqref{ChCSKJ.oqdjf34rg0hn.35} in~$[-2R,+\infty)$, we see that, for all~$x\in[-2R,-R)$,
$$ \omega v'(x)-cv''(x)= v_+*K(x)=0,$$
and the same holds true for all~$x\in(-\infty,-2R)$, thanks to~\eqref{EXT:v132r3we} and~\eqref{EXT:v132r3}.
This, combined with~\eqref{EXT:v132r3.304}, yields that~$v\in C^2(\R)$ is a solution of~\eqref{ChCSKJ.oqdjf34rg0hn.35}
in all~$\R$ and the proof of the desired result is thereby complete.
\end{proof}

\begin{corollary}\label{coISPK}
Assume~\eqref{LAMB} and that~$c=1$. Suppose that there exists~$v\in C([-2R,+\infty))$ such that, for all~$x\in[-2R,+\infty)$,
\begin{equation}\label{FIXP}
v(x)=\int_0^x e^{\omega\xi}\left(\kappa-\int_0^\xi e^{-\omega\theta} v_+*K(\theta)\,d\theta\right)\,d\xi.
\end{equation}

Then, $v(0)=0$ and~$v'(0)=\kappa$. 

Also, for all~$x\in[-2R,0]$, we have that~$v(x)\le0\le v'(x)$.

Moreover,
$v$ can be extended to a function in~$C^2(\R)$ which solves~\eqref{ChCSKJ.oqdjf34rg0hn.35} in the whole of~$\R$,
and the function~$u$ defined in~\eqref{ChCSKJ.oqdjf34rg0hn.354}
is a solution of~\eqref{MAIN:EQ:pa3-TW}.
\end{corollary}

\begin{proof} By direct inspection, we have that~$v(0)=0$ and~$v'(0)=\kappa$.

Moreover, if~$v$ solves~\eqref{FIXP}, we observe that~$v$ is twice differentiable in~$[-2R,+\infty)$, with
\begin{equation}\label{ChCSKJ} v'(x)=e^{\omega x}\left(\kappa-\int_0^x e^{-\omega\theta} v_+*K(\theta)\,d\theta\right)\end{equation}
and
\begin{equation*} v''(x)=\omega e^{\omega x}\left(\kappa-\int_0^x e^{-\omega\theta} v_+*K(\theta)\,d\theta\right)
- v_+*K(x),\end{equation*}
from which we obtain that~$v$ solves~\eqref{ChCSKJ.oqdjf34rg0hn.35} in~$[-2R,+\infty)$.

Also, as a byproduct of~\eqref{ChCSKJ}, for all~$x\in[-2R,0]$,
$$v'(x)=e^{\omega x}\left(\kappa+\int^0_x e^{-\omega\theta} v_+*K(\theta)\,d\theta\right)\ge0.$$
As a result, for all~$x\in[-2R,0]$,
$$ v(x)=v(x)-v(0)=-\int_x^0v'(\xi)\,d\xi\le0.$$

Hence, in light of Lemma~\ref{EXT:v}, we can extend~$v$ to the whole of~$\R$ satisfying~\eqref{ChCSKJ.oqdjf34rg0hn.35}.
This and Lemma~\ref{lSJD} yield the desired result.
\end{proof}

It is also interesting to observe that~\eqref{FIXP} completely identifies all the traveling waves, since,
up to suitable
translations, the expression found in~\eqref{FIXP} is the only
possible for traveling waves:

\begin{lemma} Let~$c=1$ and~$v$ be a solution of~\eqref{ChCSKJ.oqdjf34rg0hn.35} in the whole of~$\R$. 

Then, for all~$x$, $x_0\in\R$,
\begin{equation}\label{SJO0-2urjthimh6587bnOSJHND-02cdwsw2rvb3e5fe4536}
v'(x)=e^{\omega x}\left( e^{-\omega x_0} v'(x_0)-\int_{x_0}^x e^{-\omega\theta} v_+*K(\theta)\,d\theta\right)
\end{equation}
and
\begin{equation}\label{FIXP-NEC}\begin{split}
v(x)&=v(x_0)+\int_{x_0}^x e^{\omega \xi}\left( e^{-\omega x_0}v'(x_0)-\int_{x_0}^\xi e^{-\omega\theta} v_+*K(\theta)\,d\theta\right)\,d\xi\\&=v(x_0)+\frac{(e^{\omega (x-x_0)}-1)\,v'(x_0)}\omega
-\frac1\omega\int_{x_0}^x (e^{\omega(x-\theta)}-1) \,v_+*K(\theta)\,d\theta
.\end{split}\end{equation}
\end{lemma}

\begin{proof} Given~$x_0\in\R$, if
$$ W(x):=
e^{-\omega x} v'(x)-e^{-\omega x_0}v'(x_0)+\int_{x_0}^x e^{-\omega\theta} v_+*K(\theta)\,d\theta,$$
we have that~$W(x_0)=0$ and
\begin{eqnarray*}
W'(x)=e^{-\omega x} v''(x)
-\omega e^{-\omega x} v'(x)+e^{-\omega x} v_+*K(x)=0.
\end{eqnarray*}
This yields that~$W$ vanishes identically and consequently we obtain~\eqref{SJO0-2urjthimh6587bnOSJHND-02cdwsw2rvb3e5fe4536}.

Hence, after an additional integration in~\eqref{SJO0-2urjthimh6587bnOSJHND-02cdwsw2rvb3e5fe4536}
and a change of order of the integrals,
\begin{equation*}\begin{split}
v(x)&=v(x_0)+\int_{x_0}^x e^{\omega \xi}\left( e^{-\omega x_0}v'(x_0)-\int_{x_0}^\xi e^{-\omega\theta} v_+*K(\theta)\,d\theta\right)\,d\xi\\&=v(x_0)+\frac{(e^{\omega x}-e^{\omega x_0})\,e^{-\omega x_0}v'(x_0)}\omega
-\int_{x_0}^x \left( \int_{\theta}^x e^{\omega(\xi-\theta)} v_+*K(\theta)\,d\xi\right)\,d\theta\\&=v(x_0)+\frac{(e^{\omega (x-x_0)}-1)\,v'(x_0)}\omega
-\frac1\omega\int_{x_0}^x (e^{\omega(x-\theta)}-1) \,v_+*K(\theta)\,d\theta
,\end{split}
\end{equation*}as desired.
\end{proof}

\section{Proof of Theorem~\ref{TRAVE}}\label{DS5}
First of all, up to replacing~$\omega$ with~$\frac\omega{c}$ and~$K$ by~$\frac{K}c$, we can suppose that~$c=1$. 
Consequently, bearing in mind Corollary~\ref{coISPK},
to establish Theorem~\ref{TRAVE}, it suffices to find
\begin{equation}\label{12rb3eg3rged2ac342i}\begin{split}&
{\mbox{$v\in C([-2R,+\infty))$
such that~\eqref{FIXP} is satisfied for all~$x\in[-2R,+\infty)$.}}\end{split}\end{equation}

To this end, we define the functional space
\begin{equation}\label{ojfg043yijhkCOkoqLFWGERT:X} X:=\Big\{ {\mbox{$(v,w)$ with~$v$, $w\in C([-2R,+\infty))$}}
\Big\}.\end{equation}
We pick~$M>0$, to be taken conveniently large in what follows, and we endow~$X$ with the norm
$$ \| (v,w)\|:=\sup_{x\in[-2R,+\infty)}\frac{|v(x)|}{ e^{Mx}}+
\sup_{x\in[-2R,+\infty)}\frac{|w(x)|}{e^{Mx}}.$$
We observe that
\begin{equation}\label{ojfg043yijhkCOkoqLFWGERT}
{\mbox{the space~$X$ is complete.}}
\end{equation}
Indeed, if a sequence~$(v_k,w_k)$
is Cauchy in this norm, then the sequences~$v_k$ and~$w_k$ are Cauchy in~$L^\infty([-2R,\ell])$, for all~$\ell>0$,
and therefore they converge to some~$v$ and~$w$, respectively,
uniformly in~$[-2R,\ell]$ for all~$\ell>0$,
thus~$v$, $w\in C([-2R,+\infty))$. 

As a result, given~$\e>0$,
there exists~$k_\e$ such that, for all~$j$, $k\ge k_\e$,
$$ \sup_{x\in[-2R,+\infty)}\frac{|v_j(x)-v_k(x)|}{ e^{Mx}}+
\sup_{x\in[-2R,+\infty)}\frac{|w_j(x)-w_k(x)|}{e^{Mx}}\le\e$$
and consequently, for all~$x\in[-2R,+\infty)$,
\begin{eqnarray*}
&&\frac{|v(x)-v_k(x)|}{ e^{Mx}}+\frac{|w(x)-w_k(x)|}{e^{Mx}}
=\lim_{j\to+\infty}
\frac{|v_j(x)-v_k(x)|}{ e^{Mx}}+\frac{|w_j(x)-w_k(x)|}{e^{Mx}}
\le\e.
\end{eqnarray*}
This gives that, for all~$k\ge k_\e$,
we have that~$\|(v,w)-(v_k,w_k)\|\le2\e$
and the proof of~\eqref{ojfg043yijhkCOkoqLFWGERT} is complete.

Now, for all~$x\in[-2R,+\infty)$, we define
\begin{equation}\label{LASEWRGTRYqwfg0435y-3i65k}\begin{split}&
\Phi_1(v,w;x):=\int_0^x w(\xi)\,d\xi\\{\mbox{and}}\quad&
\Phi_2(v,w;x):=e^{\omega x}\left(\kappa-\int_0^x e^{-\omega\theta} v_+*K(\theta)\,d\theta\right).\end{split}
\end{equation}
We also use the short notation~$\Phi(v,w;x):=\big(\Phi_1(v,w;x),\Phi_2(v,w;x)\big)$
and denote by~${\mathcal{B}}_\rho$ the (say, closed)
ball of radius~$\rho>0$ in~$X$.

We claim that, for suitable choices of~$M$ and~$\rho$,
\begin{equation}\label{9ojfg430uyjhm:0itkjymy}
{\mbox{$\Phi:{\mathcal{B}}_\rho\to{\mathcal{B}}_\rho$ is a contraction.}}
\end{equation}
Let us postpone the proof of this claim and first show that this would lead
to the desired result in Theorem~\ref{TRAVE}.
Indeed, if~\eqref{9ojfg430uyjhm:0itkjymy} holds true, one deduces from~\eqref{ojfg043yijhkCOkoqLFWGERT}
and the Contraction Mapping Theorem that there exists a solution~$(v,w)\in{\mathcal{B}}_\rho$ of the fixed point problem~$
(v(x),w(x))=\Phi(v,w;x)$.

In particular, we have that~$v$, $w\in C([-2R,+\infty))$ and that, for all~$x\in[-2R,+\infty)$,
$$\begin{dcases}
\displaystyle v(x)=\int_0^x w(\xi)\,d\xi\\ \displaystyle 
w(x)=e^{\omega x}\left(\kappa-\int_0^x e^{-\omega\theta} v_+*K(\theta)\,d\theta\right).
\end{dcases}$$
This would lead to~\eqref{12rb3eg3rged2ac342i} and thus complete the proof of Theorem~\ref{TRAVE}.

It remains to prove~\eqref{9ojfg430uyjhm:0itkjymy}. For this objective, we use the notation
\begin{equation}\label{10.11bis}
\|f\|_\star:=\sup_{x\in[-2R,+\infty)}\frac{|f(x)|}{ e^{Mx}},\quad
x_+:=\max\{x,0\},\quad{\mbox{ and }}\quad
x_-:=\max\{-x,0\}.\end{equation}
We remark that
\begin{equation*}
\begin{split}&
\sup_{x\in[-2R,+\infty)}\frac{|\Phi_1(v,w;x)|}{ e^{Mx}}=
\sup_{x\in[-2R,+\infty)}\frac{1}{ e^{Mx}}\left|\int_0^x w(\xi)\,d\xi\right|\\&\qquad=
\sup_{x\in[-2R,+\infty)}\frac{1}{ e^{Mx}}\left|\int^{x_+}_{-x_-} w(\xi)\,d\xi\right|\\&\qquad\le
\sup_{x\in[-2R,+\infty)}\frac{1}{ e^{Mx}}\left(\int^{x_+}_{-x_-} |w(\xi)|\,d\xi\right)\\
&\qquad\le
\sup_{x\in[-2R,+\infty)}\frac{\|w\|_\star}{ e^{Mx}}\left(\int^{x_+}_{-x_-}e^{M\xi}\,d\xi\right)
\\&\qquad=
\sup_{x\in[-2R,+\infty)}\frac{\|w\|_\star
( e^{Mx_+}-e^{-Mx_-} )
}{ M e^{Mx}}.
\end{split}
\end{equation*}
Since
\begin{eqnarray*} e^{Mx_+}-e^{-Mx_-}&=&\begin{dcases}
e^{Mx}-1 & {\mbox{ if $x\in[0,+\infty)$}},\\
1-e^{Mx}& {\mbox{ if $x\in[-2R,0)$}},
\end{dcases}\\&\le&e^{M(x+2R)},\end{eqnarray*}
we conclude that
\begin{equation}\label{ojqsld329rtgir}
\begin{split}&
\sup_{x\in[-2R,+\infty)}\frac{|\Phi_1(v,w;x)|}{ e^{Mx}}\le \frac{e^{2MR}\|w\|_\star}{ M }.
\end{split}
\end{equation}

Also, by construction, for all~$(v,w)$, $(\tilde v,\tilde w)\in X$,
$$\Phi_1(v,w;x)-\Phi_1(\tilde v,\tilde w;x)=\int_0^x \big(w(\xi)-\tilde w(\xi)\big)\,d\xi
=\Phi_1(v-\tilde v,w-\tilde w;x)$$
and thus we infer from~\eqref{ojqsld329rtgir} that
\begin{equation}\label{ojqsld329rtgir:2}
\begin{split}
\sup_{x\in[-2R,+\infty)}\frac{|\Phi_1(v,w;x)-\Phi_1(\tilde v,\tilde w;x)|}{ e^{Mx}}&
=\sup_{x\in[-2R,+\infty)}\frac{|\Phi_1(v-\tilde v,w-\tilde w;x)|}{ e^{Mx}}\\&
\le \frac{e^{2MR}\|w-\tilde w\|_\star}{ M }.
\end{split}
\end{equation}

Furthermore, in virtue of~\eqref{LAMB}, for all~$\theta\in[-x_-,x_+]$,
\begin{eqnarray*}&&
|f*K(\theta)|=\left|\int_{-R}^R K(y)\,f(\theta-y)\,dy\right|\le\Lambda
\int_{-R}^{R}|f(\theta-y)|\,dy\\&&\qquad\le\Lambda
\|f\|_\star \int_{-R}^{R}e^{M(\theta-y)}\,dy=\frac{\Lambda\|f\|_\star(e^{M(\theta+R)}-e^{M(\theta-R)})}M.
\end{eqnarray*}
As a result, if~$M>\omega$,
\begin{eqnarray*}&&
\left|\int_0^x e^{-\omega\theta} f*K(\theta)\,d\theta\right|=\left|\int_{-x_-}^{x_+} e^{-\omega\theta} f*K(\theta)\,d\theta\right|\\
&&\qquad\le\int_{-x_-}^{x_+} e^{-\omega\theta} |f*K(\theta)|\,d\theta
\le\frac{ \Lambda (e^{MR}-e^{-MR})\|f\|_\star}M
\int_{-x_-}^{x_+} e^{(M-\omega)\theta}\,d\theta\\&&\qquad=
\frac{\Lambda (e^{MR}-e^{-MR})\|f\|_\star(e^{(M-\omega)x_+}-e^{-(M-\omega)x_-})}{M(M-\omega)}\\&&\qquad
\le \frac{ \Lambda e^{MR+(M-\omega)x_+}\|f\|_\star}{M(M-\omega)}.
\end{eqnarray*}

This yields that
\begin{equation}\label{02jrtmgy.10imy45a2X5t66}
\begin{split}&
\sup_{x\in[-2R,+\infty)}\frac{|\Phi_2(v,w;x)-\Phi_2(\tilde v,\tilde w;x)|}{e^{Mx}}\\&\qquad=
\sup_{x\in[-2R,+\infty)}
e^{(\omega-M) x}\left|\int_0^x e^{-\omega\theta} (v_+-\tilde v_+)*K(\theta)\,d\theta\right|\\&\qquad
\le
\frac{\Lambda e^{3MR}\|v_+-\tilde v_+\|_\star}{M(M-\omega)}=
\frac{\Lambda e^{3MR}}{M(M-\omega)}\sup_{x\in[-2R,+\infty)}\frac{|v_+(x)-\tilde v_+(x)|}{e^{Mx}}\\&\qquad\le
\frac{\Lambda e^{3MR}}{M(M-\omega)}\sup_{x\in[-2R,+\infty)}\frac{|v(x)-\tilde v(x)|}{ e^{Mx}}=\frac{\Lambda e^{3MR}\|v-\tilde v\|_\star}{M(M-\omega)}.
\end{split}
\end{equation}

It follows from this and~\eqref{ojqsld329rtgir:2} that
\begin{eqnarray*}
\| \Phi(v,w;x)-\Phi(\tilde v,\tilde w;x)\|\le
\frac{e^{2MR}\|w-\tilde w\|_\star}{ M }+
\frac{\Lambda e^{3MR}\|v-\tilde v\|_\star}{M(M-\omega)}.
\end{eqnarray*}
In particular, if~$M:=4+\omega+\Lambda$,
\begin{eqnarray*}
\| \Phi(v,w;x)-\Phi(\tilde v,\tilde w;x)\|\le
\frac{e^{3(4+\omega+\Lambda)R}\,\|(v,w)-(\tilde v,\tilde w)\|}{ 4}.
\end{eqnarray*}
Hence, if~$R$ is sufficiently small with respect to~$\omega$ and~$\Lambda$,
\begin{equation}\label{912urjthgn5KJMsd-ipef0218ritgr}
\| \Phi(v,w;x)-\Phi(\tilde v,\tilde w;x)\|\le
\frac{\|(v,w)-(\tilde v,\tilde w)\|}{ 2}.
\end{equation}

For this reason, to complete the proof of~\eqref{9ojfg430uyjhm:0itkjymy}, it remains to pick~$\rho>0$ such that
\begin{equation}\label{9ojfg430uyjhm:0itkjymy:09}
\Phi({\mathcal{B}}_\rho)\subseteq{\mathcal{B}}_\rho.\end{equation} To fulfill this goal, we use~\eqref{912urjthgn5KJMsd-ipef0218ritgr}
with~$(\tilde v,\tilde w):=(0,0)$, finding that, for all~$(v,w)\in{\mathcal{B}}_\rho$,
\begin{eqnarray*}\| \Phi(v,w;x)\|&\le&
\| \Phi(v,w;x)-\Phi(0,0;x)\|+
\| \Phi(0,0;x)\|\\&\le&
\frac{\|(v,w)\|}{ 2}+\|\kappa e^{\omega x}\|_\star\\&\le&\frac\rho2+\kappa\sup_{x\in[-2R,+\infty)} e^{-4x}\\&=&
\frac\rho2+\kappa e^{8R}.
\end{eqnarray*}
Hence, we choose~$\rho:=2\kappa e^{8R}$, whence~\eqref{9ojfg430uyjhm:0itkjymy:09} follows, as desired.~\hfill \qedsymbol

\section{Proof of Theorem~\ref{TRAVE.BIS}}\label{DS5.bBSIS}
The proof is a variation of that of Theorem~\ref{TRAVE}. We provide full details for the convenience of the reader.
The gist here is to endow the functional space~$X$ in~\eqref{ojfg043yijhkCOkoqLFWGERT:X}
with the norm
\begin{equation}\label{0q2erwjg-2rfgb093ihjr-23prohb329} \vertiii{ (v,w)}:=\sup_{x\in[-2R,+\infty)}\frac{|v(x)|}{ e^{Mx}}+
L\sup_{x\in[-2R,+\infty)}\frac{|w(x)|}{e^{Mx}},\end{equation}
with~$L>0$ and~$M>\omega>0$. We will pick~$L$ and~$M$ conveniently in what follows,
in dependence of the given~$R$ and~$\omega$.

The (say, closed)
ball of radius~$\rho>0$ in~$X$ with respect to this norm is denoted by~${\mathscr{B}}_\rho$ and, for all~$x\in\R$,
we consider~$\Phi(v,w;x):=\big(\Phi_1(v,w;x),\Phi_2(v,w;x)\big)$, with~$\Phi_1$ and~$\Phi_2$ as in~\eqref{LASEWRGTRYqwfg0435y-3i65k}.

As in Section~\ref{DS5}, our goal is to show that
\begin{equation}\label{9ojfg430uyjhm:0itkjymy:BIS}
{\mbox{$\Phi:{{\mathscr{B}}}_\rho\to{{\mathscr{B}}}_\rho$ is a contraction,}}
\end{equation}
since a fixed point of~$\Phi$ would automatically provide the desired traveling wave with~$v(0)=0$ and~$v'(0)=\kappa$.

To that effect, we deduce from~\eqref{10.11bis} and~\eqref{0q2erwjg-2rfgb093ihjr-23prohb329} that
$$\vertiii{ (v,w)}\ge\|v\|_\star\qquad {\mbox{and}}\qquad
\vertiii{ (v,w)}\ge L\|w\|_\star.$$
Thus, bearing in mind~\eqref{ojqsld329rtgir:2} and~\eqref{02jrtmgy.10imy45a2X5t66}, we see that
\begin{eqnarray*}&&
\vertiii{ (\Phi(v,w;x)-\Phi(\tilde v,\tilde w;x))}\\&&\quad=\sup_{x\in[-2R,+\infty)}\frac{|\Phi_1(v,w;x)-\Phi_1(\tilde v,\tilde w;x)|}{ e^{Mx}}+
L\sup_{x\in[-2R,+\infty)}\frac{|\Phi_2(v,w;x)-\Phi_2(\tilde v,\tilde w;x)|}{e^{Mx}}\\&&\quad\le
\frac{e^{2MR}\|w-\tilde w\|_\star}{ M }+
\frac{L\Lambda e^{3MR}\|v-\tilde v\|_\star}{M(M-\omega)}\\&&\quad\le
\frac{e^{3MR}}{ M }\left( \|w-\tilde w\|_\star+
\frac{L\Lambda \|v-\tilde v\|_\star}{M-\omega}\right)\\&&\quad\le
\frac{e^{3MR}}{ M }\left( \frac1L+
\frac{L\Lambda }{M-\omega}\right)\vertiii{ (v-\tilde v,w-\tilde w)}.
\end{eqnarray*}
We now choose~$M:=\omega+1$ and then~$L:=\frac{4e^{3MR}}{M}$, concluding that
\begin{equation}\label{01e2wrofhgb 01326t3rteted1ashryttg3uri}\begin{split}
\vertiii{ (\Phi(v,w;x)-\Phi(\tilde v,\tilde w;x))}&\le
\left( \frac14+
\frac{4e^{6(\omega+1)R}\Lambda }{(\omega+1)^2}\right)\vertiii{ (v-\tilde v,w-\tilde w)}\\&\le
\frac{ \vertiii{ (v-\tilde v,w-\tilde w)}}2,
\end{split}
\end{equation}
as long as~$\Lambda$ is suitably small.

Additionally, by~\eqref{01e2wrofhgb 01326t3rteted1ashryttg3uri}, if~$(v,w)\in{{\mathscr{B}}}_\rho$
and~$L$ and~$M$ are as above,
\begin{eqnarray*}
\vertiii{ \Phi(v,w;x)}&\le&\vertiii{ \Phi(v,w;x)-\Phi(0,0;x)}+\vertiii{\Phi(0,0;x)}\\&\le&
\frac{ \vertiii{ (v,w)}}2+\kappa L\sup_{x\in[-2R,+\infty)}e^{-x}\\&\leq&
\frac{ \rho}2+\kappa L e^{2R}\\&=&\rho
\end{eqnarray*}
with~$\rho:=2\kappa L e^{2R}$.

Thanks to this and~\eqref{01e2wrofhgb 01326t3rteted1ashryttg3uri},
the proof of~\eqref{9ojfg430uyjhm:0itkjymy:BIS} is thereby complete.~\hfill \qedsymbol

\section{Proof of Theorem~\ref{ABBOU}}\label{ajopsclnFG}
As observed in Section~\ref{DS5}, without loss of generality we can suppose that~$c=1$
and, by Lemma~\ref{lSJD}, we know that~$v$ is a solution of~\eqref{ChCSKJ.oqdjf34rg0hn.35}.

Moreover, without loss of generality, we can assume that
\begin{equation}\label{WELWVDFB}
{\mbox{neither~$K$ nor~$v_+$ vanish identically,}}
\end{equation}
otherwise~\eqref{FIXP-NEC} would boil down to~$v(x)=v(x_0)+\frac{(e^{\omega (x-x_0)}-1)\,v'(x_0)}\omega$, which is an unbounded function of~$x$ unless it is constant, yielding
the desired result in Theorem~\ref{ABBOU}.

As a byproduct of~\eqref{WELWVDFB}, we can find~$\varrho>0$ such that
\begin{equation}\label{ojsqdskncd932V0}
c_\varrho:=\int_{-\varrho}^{\varrho}K(y)\,dy>0.
\end{equation}

Now, suppose that~$u$ is bounded (and therefore~$v$ is bounded as well). We distinguish two cases: either, for all~$x_0\in\R$,
\begin{equation}\label{ojsqdskncd932V}e^{-\omega x_0}
v'(x_0)=\int_{x_0}^{+\infty} e^{-\omega\theta} \,v_+*K(\theta)\,d\theta
\end{equation}
or there exists~$x_0\in\R$ for which~\eqref{ojsqdskncd932V} is violated.

Let us deal first with the case in which~\eqref{ojsqdskncd932V} holds true (we will actually
show that this leads to a contradiction, hence this case can be ruled out).
In this case, we have that~$v'\ge0$ and therefore, if~$v$ is bounded, it possesses two horizontal asymptotes at~$\pm\infty$. In fact, by~\eqref{WELWVDFB}, this gives that
\begin{equation*}
\ell:=\lim_{x\to+\infty}v(x)>0.
\end{equation*}
Thus, we find~$\overline{x}$ such that for all~$x\ge\overline{x}$ we have that~$v(x)\ge\frac\ell2$.

Using this, \eqref{ojsqdskncd932V0}, and~\eqref{ojsqdskncd932V}, we conclude that, for all~$x\ge\overline{x}+\varrho$,
\begin{eqnarray*}v'(x)&=&\int_{x}^{+\infty} e^{\omega(x-\theta)} \,v_+*K(\theta)\,d\theta
\\&\ge&\int_{x}^{+\infty} \left(\int_{-\varrho}^{\varrho} e^{\omega(x-\theta)} v_+(\theta-y)\,K(y)\,dy\right)\,d\theta\\&\ge&\frac\ell2\int_{x}^{+\infty} \left(\int_{-\varrho}^{\varrho} e^{\omega(x-\theta)} K(y)\,dy\right)\,d\theta\\&\ge&\frac{c_\varrho\,\ell}2\int_{x}^{+\infty} e^{\omega(x-\theta)}\,d\theta\\&=&\frac{c_\varrho\,\ell}{2\omega}.
\end{eqnarray*}
But then
\begin{eqnarray*}&&\lim_{x\to+\infty}v(x)= v(\overline{x}+\varrho)+\lim_{x\to+\infty}\int_{\overline{x}+\varrho}^x v'(\xi)\,d\xi\\&&\qquad
\ge  v(\overline{x}+\varrho)+\lim_{x\to+\infty}\frac{c_\varrho\,\ell\,(x-\overline{x}-\varrho)}{2\omega}=+\infty,\end{eqnarray*}
in contradiction with the assumption that~$v$ is bounded.

Let us now consider the case in which there exists~$x_0\in\R$ such that~\eqref{ojsqdskncd932V} is violated, namely
\begin{equation}\label{0i2orjfegrb9jSDM}e^{-\omega x_0}
v'(x_0)-\int_{x_0}^{+\infty} e^{-\omega\theta} \,v_+*K(\theta)\,d\theta\ne0.
\end{equation}
Then, the boundedness of~$v$ and~\eqref{FIXP-NEC} give that
\begin{eqnarray*}
0&=&\lim_{x\to+\infty}\frac{v(x)-v(x_0)}{e^{\omega x}}\\&=&\lim_{x\to+\infty}
\left[ \frac{(e^{-\omega x_0}-e^{-\omega x})\,v'(x_0)}\omega
-\frac1\omega\int_{x_0}^x (e^{-\omega\theta}-e^{-\omega x}) \,v_+*K(\theta)\,d\theta\right]\\&=&
\frac1\omega
\left[ e^{-\omega x_0}v'(x_0)
-\lim_{x\to+\infty}\int_{x_0}^x (e^{-\omega\theta}-e^{-\omega x}) \,v_+*K(\theta)\,d\theta\right].
\end{eqnarray*}
Hence, by the Dominated Convergence Theorem,
$$ 0=\frac1\omega
\left[ e^{-\omega x_0}v'(x_0)
-\int_{x_0}^{+\infty} e^{-\omega\theta} \,v_+*K(\theta)\,d\theta\right],
$$but this is in contradiction with~\eqref{0i2orjfegrb9jSDM} and the proof of Theorem~\ref{ABBOU} is thereby complete.~\hfill \qedsymbol

\section{Proof of Theorem~\ref{NONMSPT43}}\label{N56789-97ONMSPT43:OSIDNKF3495TI:2345T0}
First of all, we point out that, by~\eqref{ChCSKJ.oqdjf34rg0hn.35},
\begin{equation*}
-\frac{d}{dx} \left( e^{-\omega x}v'(x)\right)=
e^{-\omega x}\big(
\omega v'(x)-v''(x)\big)=e^{-\omega x}\, v_+*K(x).
\end{equation*}
In particular, 
\begin{equation}\label{0jghmlytni8yX989}
\frac{d}{dx} \left( e^{-\omega x}v'(x)\right)\le0.\end{equation}
Therefore, for all~$x\le0$, integrating~\eqref{0jghmlytni8yX989} over the segment~$[x,0]$ we find that
\begin{equation*}
\kappa-e^{-\omega x}v'(x)\le0.
\end{equation*}
Similarly, for all~$x\ge0$, integrating~\eqref{0jghmlytni8yX989} over the segment~$[0,x]$ we find that
$$e^{-\omega x}v'(x)-\kappa\le0$$
and these observations establish~\eqref{80iujgo486nb7m2985vb65BSAnmgyhn9}.

Let us now assume~\eqref{LAMBdasotto} and suppose, for the sake of contradiction, that~\eqref{LAMBdasotto.bi}
holds true. Then, for all~$x\ge0$,
$$ v(x)\ge\frac{\kappa_\star}{\omega}(e^{\omega x}-1)$$
and therefore, for all~$x\ge\varrho$,
\begin{eqnarray*}&& v_+*K(x)\ge\lambda\int_{-\varrho}^{\varrho} v_+(x-y)\,dy
\ge\frac{\lambda\kappa_\star}{\omega}\int_{-\varrho}^{\varrho} (e^{\omega (x-y)}-1)\,dy\\&&\quad=
\frac{\lambda\kappa_\star}{\omega}\left( 
\frac{ (e^{\omega\varrho}-e^{-\omega\varrho})\,
e^{\omega x}}\omega
-2\varrho\right).\end{eqnarray*}

Combining this and~\eqref{SJO0-2urjthimh6587bnOSJHND-02cdwsw2rvb3e5fe4536}, we gather that,
when~$x\ge\varrho$,
\begin{eqnarray*}
e^{-\omega x}v'(x)&=&\kappa-\int_{0}^x e^{-\omega\theta} v_+*K(\theta)\,d\theta \\&\le&
\kappa-\frac{\lambda\kappa_\star}{\omega}
\int_{0}^x \left( 
\frac{e^{\omega\varrho}-e^{-\omega\varrho}}\omega
-2\varrho e^{-\omega\theta} \right)
\,d\theta\\&=&
\kappa-\frac{\lambda\kappa_\star}{\omega^2}\Big( 
(e^{\omega\varrho}-e^{-\omega\varrho})x
-2\varrho  (1-e^{-\omega x})\Big).
\end{eqnarray*}
Hence, since
$$ \lim_{x\to+\infty}(e^{\omega\varrho}-e^{-\omega\varrho})x
-2\varrho  (1-e^{-\omega x})=+\infty,$$
we infer that, if~$x$ is sufficiently large, then~$e^{-\omega x}v'(x)<0$,
but this is in contradiction with~\eqref{LAMBdasotto.bi}.~\hfill \qedsymbol

\section{Proof of Theorem~\ref{ABBOU:side}}\label{oqwrufjgr0tphknwfbn}
Without loss of generality, we suppose that~$c=1$.
As a byproduct of~\eqref{80iujgo486nb7m2985vb65BSAnmgyhn9},
we have that, for every~$x\in(-\infty,0]$,
\begin{equation*}
v(x)=v(x)-v(0)=-\int_{x}^0 v'(\tau)\,d\tau\le-\kappa\int_{x}^0 e^{\omega\tau}\,d\tau=-\frac{\kappa}{\omega}(1-e^{\omega x})\le0.
\end{equation*}
Hence, for all~$x\in(-\infty,-R)$,
$$ v_+*K(x)=\int_{-R}^R v_+(x-y)\,K(y)\,dy=0$$
and therefore, by means of~\eqref{FIXP-NEC}, for all~$x\in(-\infty,-R)$,
\begin{equation*}
v(x)=v(-R)+\frac{(e^{\omega (x+R)}-1)\,v'(-R)}\omega
,\end{equation*}
from which~\eqref{SIDE-la1} plainly follows.

This and Theorem~\ref{ABBOU} yield~\eqref{SIDE-la2}.~\hfill \qedsymbol

\section{Proof of Theorem~\ref{NONMSPT43.rm}}\label{2efm02emiaksMEQWfvqw01367vf2-323r} 
We argue by contradiction and suppose that~$v_\omega$
is nondecreasing for all~$\omega$ arbitrarily small (and below we will
implicitly suppose that~$\omega\le1$). In particular, $v_\omega(x)\ge
v_\omega(0)=0$ for all~$x\in[0,+\infty)$ and~$v_\omega(x)\le
v_\omega(0)=0$ for all~$x\in(-\infty,0]$.

Also, it follows from~\eqref{80iujgo486nb7m2985vb65BSAnmgyhn9} that, for all~$x\in[0,+\infty)$,
\begin{eqnarray*}&&v_\omega(x)=\int_0^x v_\omega'(\theta)\,d\theta\le\kappa\int_0^x e^{\omega \theta}\,d\theta=\frac{\kappa}\omega
(e^{\omega x}-1)=\kappa\sum_{j=1}^{+\infty} \frac{\omega^{j-1}x^j}{j!}\\&&\quad\qquad\qquad
=\kappa x\sum_{i=0}^{+\infty} \frac{\omega^{i}x^i}{(i+1)!}\le
\kappa x\sum_{i=0}^{+\infty}\frac{x^i}{i!}=\kappa x e^{ x}
.\end{eqnarray*}

Consequently, $v_\omega$ is bounded in all compact subsets of~$[0,+\infty)$, uniformly in~$\omega$:
namely, for all~$\ell>0$,
$$ \sup_{{x\in[0,\ell]}\atop{\omega\in(0,1]}} |v_\omega(x)|\le \kappa \ell e^{ \ell}.$$

As a result, by~\eqref{LAMB} and~\eqref{SJO0-2urjthimh6587bnOSJHND-02cdwsw2rvb3e5fe4536}, for all~$x\in[0,\ell]$,
\begin{eqnarray*}
|v_\omega'(x)|&=&e^{\omega x}\left| \kappa-\int_{0}^x \left( \int_\R e^{-\omega\theta} v_{\omega,+}(\theta-y)\,K(y)\,dy\right)\,d\theta\right|\\&\le&e^{\omega \ell}\left( \kappa+
\Lambda\int_{0}^\ell \left( \int_{-R}^R e^{-\omega\theta}\kappa(\ell+R)e^{\ell+R}\,dy\right)\,d\theta\right)\\&\le&e^{ \ell}\Big( \kappa+
2\kappa\Lambda R\ell (\ell+R)e^{\ell+R}\Big)\\&=:&C_\ell,
\end{eqnarray*}
showing that also~$v_\omega'$ is bounded in all compact subsets of~$[0,+\infty)$, uniformly in~$\omega$.

Moreover, by~\eqref{ChCSKJ.oqdjf34rg0hn.35}, if~$x\in[0,\ell]$,
\begin{eqnarray*}
|v''_\omega(x)|&\le&| v'_\omega(x)|+|v_{\omega,+}*K(x)|\\&\le&
C_\ell+\Lambda\int_{-R}^R |v_{\omega,+}(x-y)|\,dy\\&\le&
C_\ell+2\Lambda R\sup_{[0,\ell+R] }|v_{\omega}|\\&\le&
C_\ell+2 \kappa\Lambda R (\ell+R) e^{ \ell+R}.
\end{eqnarray*}
This and~\eqref{ChCSKJ.oqdjf34rg0hn.35} yield that~$v_\omega''$ is bounded in all compact subsets of~$[0,+\infty)$, uniformly in~$\omega$.

Hence, we can extract a (not relabeled) sequence such that~$v_\omega$ and its derivative converge uniformly
in all sets of the form~$[0,\ell]$. By construction, denoting~$v_0:[0,+\infty)\to\R$ this limit function,
we have that~$v_0(0)=0$, $v_0'(0)=\kappa>0$ and~$v_0$ is nondecreasing.

Hence, bearing in mind~\eqref{LAMBdasotto} and~\eqref{FIXP-NEC},
\begin{eqnarray*}
v_0(x)&=&
\lim_{\omega\searrow0}v_\omega(x)\\&=&\lim_{\omega\searrow0}\left(\frac{(e^{\omega x}-1)\,\kappa}\omega
-\frac1\omega\int_{0}^x (e^{\omega(x-\theta)}-1) \,v_{\omega,+}*K(\theta)\,d\theta\right)\\&=&
\kappa x
-\int_{0}^x (x-\theta)\,v_{0,+}*K(\theta)\,d\theta\\&\le&
\kappa x
-\lambda\int_{0}^x\left(\int_{-\varrho}^\varrho (x-\theta)\,v_{0,+}(\theta-y)\,dy\right)\,d\theta.
\end{eqnarray*}

Accordingly, using the monotonicity of~$v_0$, we see that if~$\theta\ge2\varrho$ and~$y\le\varrho$ then~$v_0(\theta-y)\ge v_0(\varrho)>0$
and thus
\begin{eqnarray*}
0&<& \lim_{x\to+\infty}v_0(x)\\&\le&
\lim_{x\to+\infty}\left[
\kappa x
-\lambda\int_{2\varrho}^x\left(\int_{-\varrho}^\varrho (x-\theta)\,v_{0,+}(\theta-y)\,dy\right)\,d\theta\right]\\& \le&
\lim_{x\to+\infty}\left[
\kappa x
-\lambda v_0(\varrho)\int_{2\varrho}^x\left(\int_{-\varrho}^\varrho (x-\theta)\,dy\right)\,d\theta\right]\\& =&
\lim_{x\to+\infty}\left[
\kappa x
-\lambda v_0(\varrho)\varrho( x-2\varrho)^2\right]\\&=&-\infty,
\end{eqnarray*} which is a contradiction.~\hfill \qedsymbol

\section{Proof of Theorem~\ref{qdwsfvbolgrb:SDVb}}\label{qdwsfvbolgrb:SDVb:SEC}
Let us suppose, without loss of generality, that~$c=1$.
We deduce from~\eqref{80iujgo486nb7m2985vb65BSAnmgyhn9} that
\begin{equation*}
{\mbox{and $v'(x)>0$ for all~$x\in(-\infty,0]$}}\end{equation*}
and therefore
\begin{equation}\label{10wdfojv0pkfvwpiknfvd}
{\mbox{$v(x)\le0$ for all~$x\in(-\infty,0]$.}}\end{equation}

Also, by means of~\eqref{80iujgo486nb7m2985vb65BSAnmgyhn9}, for all~$x\in[0,+\infty)$,
$$ v(x)=v(x)-v(0)=\int_0^x v'(y)\,dy\le
\kappa \int_0^x e^{\omega y}\,dy=\frac{\kappa(e^{\omega x}-1)}\omega.$$
Consequently, for all~$x\in[0,+\infty)$,
$$ v_+(x)\le\frac{\kappa(e^{\omega x}-1)}\omega.$$
Hence, we use~\eqref{LAMB} and~\eqref{10wdfojv0pkfvwpiknfvd}, finding that, for all~$x\in[0,+\infty)$,
\begin{equation*}\begin{split}
v_+*K(x)&\le\Lambda\int_{-R}^R v_+(x-y)\,dy\\&=\Lambda\int_{-R}^{\min\{R,x\}} v_+(x-y)\,dy\\&\le
\frac{\Lambda\kappa}\omega
\int_{-R}^{\min\{R,x\}} 
(e^{\omega (x-y)}-1)\,dy\\&=
\frac{\Lambda\kappa}{\omega^2}(e^{\omega(x+R)}-e^{\omega\max\{x-R,0\}})
-\frac{\Lambda\kappa}\omega(\min\{R,x\}+R)\\&\le\frac{\Lambda\kappa}{\omega^2}(e^{\omega(x+R)}-e^{\omega\max\{x-R,0\}}).\end{split}
\end{equation*}

As a result, by Lemma~\ref{lSJD}
and equation~\eqref{SJO0-2urjthimh6587bnOSJHND-02cdwsw2rvb3e5fe4536}, for all~$x\in[0,+\infty)$,
\begin{equation*}
\begin{split}
\kappa-e^{-\omega x}v'(x)&=\int_0^x e^{-\omega\theta} v_+*K(\theta)\,d\theta\\&\le
\frac{\Lambda\kappa}{\omega^2}\int_0^x
(e^{\omega(\theta+R)}-e^{\omega\max\{\theta-R,0\}})
\,d\theta
\end{split}\end{equation*}
and therefore
\begin{equation}\label{QdwfvEasxcvfdWTEGRHNevf} \frac{e^{-\omega x}v'(x)}\kappa\ge1- \Phi_{\omega,\Lambda,R}(x),\end{equation}
where
$$ \Phi_{\omega,\Lambda,R}(x):=\frac{\Lambda}{\omega^2}\int_0^x
(e^{\omega(\theta+R)}-e^{\omega\max\{\theta-R,0\}})
\,d\theta.$$
Notice that~$ \Phi_{\omega,\Lambda,R}$ is continuous in~$[0,+\infty)$, hence there exists~$L>0$,
depending only on~$\omega$, $\Lambda$, and~$R$, such that, for all~$x\in[0,L)$ we have that~$\Phi_{\omega,\Lambda,R}(x)<1$.

Recalling~\eqref{QdwfvEasxcvfdWTEGRHNevf} we thereby conclude that, for all~$x\in[0,L)$,
we have that~$\frac{e^{-\omega x}v'(x)}\kappa>0$, and thus~$v'(x)>0$, as desired.

In addition, when~$\omega=1$, 
\begin{eqnarray*} \Phi_{\omega,\Lambda,R}(x)&=&
\Lambda\int_0^x (e^{\theta+R}-e^{\max\{\theta-R,0\}})\,d\theta\\&=&\begin{dcases}\displaystyle\Lambda\left(
e^R(e^x-1)-x\right), & {\mbox{ if }}\displaystyle x \in[0,R], \\ \displaystyle\Lambda\left(
e^x(e^R-e^{-R})-e^R-R+1\right), & {\mbox{ if }}\displaystyle x \in(R,+\infty).
\end{dcases}\end{eqnarray*}

In particular, if~$x\in[0,R]$ and~$\Lambda\left(
e^R(e^R-1)-R\right)<1$, then
$$ \Phi_{\omega,\Lambda,R}(x)\le\Phi_{\omega,\Lambda,R}(R)=
\Lambda\left(
e^R(e^R-1)-R\right)<1.$$
Similarly, when~$L_\star:=\ln\left(\frac{1+\Lambda(e^R+R-1)}{\Lambda(e^R-e^{-R})}\right)$
and~$x\in(R,L_\star)$, we see that
$$ \Phi_{\omega,\Lambda,R}(x)<\Phi_{\omega,\Lambda,R}(L_\star)=1.$$
These observations give~\eqref{LESYT:Sdwoed}.~\hfill \qedsymbol

\section{Proof of Theorem~\ref{sojdcmvvb-rw-SeeDeqwdf}}\label{sojdcmvvb-rw-SeeDeqwdf:S}

Up to a vertical translation of~$v$, we can assume that~$\Theta=0$.
In this scenario, equation~\eqref{PASD:erf} reduces to
\begin{eqnarray*} \dot\lambda (t)\,{x}\,v'\left( \lambda(t)\,x\right)&
=& \partial_t u(x,t)\\&=&c\Delta u(x,t)+\int_{x-\ell}^{x+\ell} u_+(y,t)\,dy\\&
=&c\lambda^2(t)\,v''\left( \lambda(t)\,x\right)+
\int_{x-\ell}^{x+\ell} v_+\left( \lambda(t)\,y\right)\,dy\\&=&c\lambda^2(t)\,v''\left( \lambda(t)\,x\right)+\frac1{\lambda(t)}
\int_{\lambda(t)x-\lambda(t)\ell}^{\lambda(t)x+\lambda(t)\ell} v_+(\eta)\,d\eta,
\end{eqnarray*}as long as~$\lambda(t)\ne0$.

That is, for all~$r\in\R$ and~$t\in(0,T)$ for which~$\lambda(t)\ne0$,
\begin{equation}\label{CDWFrbul0jeme}
\dot\lambda (t)\,{r}\,v'(r)=
c\lambda^3(t)\,v''(r)+
\int_{r-\lambda(t)\ell}^{r+\lambda(t)\ell} v_+(\eta)\,d\eta.
\end{equation}

Now we claim that, for all~$r\in\R$,
\begin{equation}\label{WEQDUCUTe}
v(r)\le0.
\end{equation}
Indeed, we define
$$ \bar r:=\sup\big\{ \rho\in\R {\mbox{ s.t. }}v\le0{\mbox{ in }}(-\infty,\rho]\big\}$$
and note that~$\bar r\ge r_0$. Thus, if~\eqref{WEQDUCUTe} is violated, then~$\bar r<+\infty$
and we can find a sequence~$r_k>\bar r$, with~$v(r_k)>0$, such that~$r_k\to\bar r$ as~$k\to+\infty$.

We now distinguish two cases: either~$\lambda$ is constantly equal to zero, or~$\lambda$ is a non-constant function.

If~$\lambda$ is identically equal to zero, then~$u$ is constant, hence we can also consider~$v$ to be constant, say equal to some~$\bar v$. 
In this case, we obtain from~\eqref{PASD:erf} that
$$0=\int_{x-\ell}^{x+\ell}\bar v_+\,dy=2\ell\bar v_+,$$
giving~\eqref{WEQDUCUTe}, against the contradictory assumption.

Now, suppose that~$\lambda$ is non-constant. Then, there exists a compact (non-trivial) interval~$I$ such that~$\lambda(t)\ne0$ and~$\dot\lambda(t)\ne0$
for all~$t\in I$. Up to replacing~$v(r)$ with~$v(-r)$ and~$\lambda$ with~$-\lambda$,
we can suppose that~$\lambda(t)>0$ for all~$t\in I$.

Consequently, 
denoting by~$J=[J_{\min},J_{\max}]$ the image of~$I$ via the map~$\lambda$, we have that,
given any~$\mu\in J$, there exists a unique~$t_\mu\in I$ such that~$\mu=\lambda(t_\mu)$.
By construction, $J_{\min}>0$.

Now,
for all~$t\in I$ and~$r\in(-\infty,\bar r-\lambda(t)\ell]$, we deduce from~\eqref{CDWFrbul0jeme} that
$$ \dot\lambda (t)\,r\,v'(r)=
c\lambda^3(t)\,v''(r).$$
Therefore, by separating variables, there exists~$a\in\R$ such that
\begin{equation*}
{\mbox{$\dot\lambda(t)=-ac\lambda^3(t)$
for all~$t\in I$.}}\end{equation*}

Hence, for all~$r\in\R$ and~$t\in I$,
we can write~\eqref{CDWFrbul0jeme} in the form
\begin{equation*}
-ac\lambda^3 (t)\,r\,v'(r)=
c\lambda^3(t)\,v''(r)+
\int_{r-\lambda(t)\ell}^{r+\lambda(t)\ell} v_+(\eta)\,d\eta.
\end{equation*}
In particular, choosing~$t:=t_\mu$, we conclude that, for all~$\mu\in J$ and~$r\in\R$,
\begin{equation*}
-ac\mu^3\,r\,v'(r)=
c\mu^3\,v''(r)+
\int_{r-\mu\ell}^{r+\mu\ell} v_+(\eta)\,d\eta
\end{equation*}
and therefore
\begin{equation}\label{K8hdSP:UmjaJSMD3IK24}
-ac\,r\,v'(r)=
c\,v''(r)+\frac1{\mu^3}
\int_{r-\mu\ell}^{r+\mu\ell} v_+(\eta)\,d\eta.
\end{equation}
Differentiating in~$\mu$, we get that, for all~$\mu\in J$ and~$r\in\R$,
\begin{equation*}
0=-\frac3{\mu^4}
\int_{r-\mu\ell}^{r+\mu\ell} v_+(\eta)\,d\eta
+\frac\ell{\mu^3}\left( v_+(r+\mu\ell)+v_+(r-\mu\ell)\right).
\end{equation*}

We specialize this identity by choosing~$r:=r_k-\mu\ell$, gathering that, for all~$\mu\in J$,
\begin{equation}\label{SPEKLMFbwkC2O}
0=-\frac3{\mu^4}
\int_{r_k-2\mu\ell}^{r_k} v_+(\eta)\,d\eta
+\frac\ell{\mu^3}\left( v(r_k)+v_+(r_k-2\mu\ell)\right).
\end{equation}
We now point out that, for all~$\mu\in J$,
$$ \lim_{k\to+\infty}r_k-2\mu\ell=\bar r-2\mu\ell<\bar r,
$$
hence, for large~$k$, we have that~$r_k-2\mu\ell<\bar r$ and thus~$v_+(r_k-2\mu\ell)=0$.

This observation and~\eqref{SPEKLMFbwkC2O} entail that
$$0=-\frac3{\mu^4}
\int_{\bar r}^{r_k} v_+(\eta)\,d\eta
+\frac{\ell v(r_k)}{\mu^3}.$$
Thus, dividing by~$r_k-\bar r$ and taking the limit as~$k\to+\infty$,
$$0=-\frac{3 v_+(\bar r)}{\mu^4}
+\frac{\ell v'(\bar r)}{\mu^3}=\frac{\ell v'(\bar r)}{\mu^3},
$$
leading to
\begin{equation}\label{PKSM-vpor}
v'(\bar r)=0.\end{equation}

We now claim that
\begin{equation}\label{DEAjor2j043jym:asj0mr4t}
v''(\bar r)\ge0.
\end{equation}
To check this, we let~$\e_k:=r_k-\bar r$ and we use~\eqref{PKSM-vpor} to see that
$$ 0<v(r_k)=v(\bar r+\e_k)=v(\bar r+\e_k)-v(\bar r)-v'(\bar r)\e_k
=\int_0^{\e_k} \big(v'(\bar r+\tau)-v'(\bar r)\big)\,d\tau,
$$
hence~\eqref{DEAjor2j043jym:asj0mr4t} follows by dividing by~$\e_k^2$ and taking the limit as~$k\to+\infty$.

As a result, evaluating~\eqref{K8hdSP:UmjaJSMD3IK24} at~$r:=\bar r$ and using~\eqref{DEAjor2j043jym:asj0mr4t},
$$0=
c\,v''(\bar r)+\frac1{\mu^3}
\int_{\bar r-\mu\ell}^{\bar r+\mu\ell} v_+(\eta)\,d\eta\ge \int_{\bar r-\mu\ell}^{\bar r+\mu\ell} v_+(\eta)\,d\eta.$$
On this account, we have that~$v\le0$ in~$(-\infty,\bar r+\mu\ell]$, violating the definition of~$\bar r$.
The proof of~\eqref{WEQDUCUTe} is thereby complete.

It follows from~\eqref{WEQDUCUTe} that~\eqref{PASD:erf}
reduces to the heat equation and thus the solution is given by~\eqref{PLOTFORFI2}.~\hfill \qedsymbol

%
%
%

\section*{Conclusion}

This study clarifies the propagation mechanisms permitted by the bushfire model 
proposed in~\cite{PAPER1}
and identifies the parameter regimes that govern spreading versus extinction. By deriving explicit bounds on the environmental diffusion and ignition kernels, we identify precise thresholds separating global propagation from natural decay, thereby sharpening the qualitative picture of fire dynamics in this framework.

A central structural outcome of the analysis is the non-existence of vertically translating solutions. In contrast, traveling fronts emerge robustly: for ignition kernels of sufficiently limited range or intensity, waves with arbitrary prescribed speeds are always admissible. These fronts necessarily develop unbounded profiles, reflecting the intrinsic growth mechanisms built into the model.
A notable feature is that wave propagation is not externally driven (for instance, by wind forcing), but is instead generated endogenously by the temperature profile itself. In particular, propagation effects arise from steep temperature gradients, including those occurring in spatially remote regions through the ignition kernel. The motion of the front is therefore a self-organised phenomenon, determined by the internal thermal structure of the wave rather than by an imposed advective mechanism.

\vfill

\end{document}